\documentclass[a4paper,11pt]{amsart}
\usepackage{amsfonts,amssymb,amsmath,amsthm,abstract,color}
\usepackage[mathscr]{euscript}
\usepackage[ps,all,arc,rotate]{xy}
\usepackage[lmargin=1in,rmargin=1in,tmargin=1in,bmargin=1in]{geometry}
\usepackage{fancyhdr}
\usepackage{dsfont}
\usepackage{pb-diagram}
\usepackage{enumitem}
\usepackage{hyperref}
\usepackage[usenames,dvipsnames]{xcolor}
\usepackage{tikz-cd}
\usepackage{extpfeil}
\usepackage{float}
\hypersetup{colorlinks=true,citecolor=NavyBlue,linkcolor=Brown,urlcolor=Orange}

\newtheorem{prop}{Proposition}[section]
\newtheorem{lemma}[prop]{Lemma}
\newtheorem{thm}[prop]{Theorem}
\newtheorem{cor}[prop]{Corollary}
\newtheorem{conj}[prop]{Conjecture}

\theoremstyle{definition}
\newtheorem{defn}[prop]{Definition}

\newtheorem{rmk}[prop]{Remark}

\renewcommand{\tilde}{\widetilde}

\DeclareMathOperator{\Alb}{Alb}

\DeclareMathOperator{\rk}{rk}     
\DeclareMathOperator{\Pic}{Pic} 
\DeclareMathOperator{\spec}{Spec} 
\DeclareMathOperator{\Spec}{Spec} 
\DeclareMathOperator{\Sym}{Sym}

\DeclareMathOperator{\Bl}{Bl}

\DeclareMathOperator{\Hom}{Hom}

\DeclareMathOperator{\id}{id} 
\DeclareMathOperator{\Mod}{Mod}

\DeclareMathOperator{\CH}{CH}
\DeclareMathOperator{\CHM}{CHM}
\DeclareMathOperator{\Mot}{Mot}

\newcommand{\et}{\mathrm{\acute{e}t}}

\renewcommand{\top}{\mathrm{top}}

\DeclareMathOperator{\pt}{pt}

\DeclareMathOperator{\im}{Im}
\DeclareMathOperator{\PP}{\mathbb{P}}

\DeclareMathOperator{\C}{\mathbb{C}}
\DeclareMathOperator{\R}{\mathbb{R}}
\DeclareMathOperator{\Z}{\mathbb{Z}}
\DeclareMathOperator{\Q}{\mathbb{Q}}

\DeclareMathOperator{\F2}{\mathbb{F}_2}

\DeclareMathOperator{\h}{\mathfrak{h}}
\DeclareMathOperator{\cl}{cl}
\DeclareMathOperator{\KR}{KR}
\DeclareMathOperator{\KU}{KU}
\DeclareMathOperator{\conjug}{conj}
\DeclareMathOperator{\Hilb}{Hilb}
\DeclareMathOperator{\Tors}{Tors}
\newcommand{\cart}{\ar@{}[dr]|\square}

\title{Maximal real varieties from moduli constructions}

\author{Lie Fu}

\thanks{The author is supported by the University of Strasbourg Institute for Advanced Study (USIAS) and by the Agence Nationale de la Recherche (ANR), under project numbers ANR-16-CE40-0011 and ANR-20-CE40-0023. }

\begin{document}

\maketitle

\begin{abstract}
For a complex manifold equipped with an anti-holomorphic involution, which is referred to as a real variety,  the Smith--Thom inequality states that the total $\F2$-Betti number of the real locus is not greater than the total $\F2$-Betti number of the ambient complex manifold.  A real variety is called maximal if the equality holds. In this paper, we present a series of new constructions of maximal real varieties by exploring moduli spaces of certain objects on a maximal real variety. Our results establish the maximality of the following real varieties:
\begin{itemize}
	\item Moduli spaces of stable vector bundles of coprime rank and degree over a maximal smooth projective real curve (known as Brugall\'e--Schaffhauser's theorem \cite{BrugalleSchaffhauser}, with a short new proof presented in this work); the same result holds for moduli spaces of stable parabolic vector bundles.
	\item Moduli spaces of stable Higgs bundles of coprime rank and degree over a maximal smooth projective real curve, providing maximal hyper-K\"ahler examples.
	\item If a real variety has non-empty real locus and maximal Hilbert square, then the variety itself and its Hilbert cube are maximal. This is always the case for maximal real smooth cubic threefolds, but never the case for maximal real smooth cubic fourfolds.
	\item Punctual Hilbert schemes on a maximal real projective surface with vanishing first $\F2$-Betti number and connected real locus, such as $\R$-rational maximal real surfaces and some generalized Dolgachev surfaces.
	\item Moduli spaces of stable sheaves on the real projective plane, or more generally, on an $\R$-rational maximal Poisson surface.
\end{itemize} 
We also observe that maximality is a motivic property when interpreted as equivariant formality. Furthermore, any smooth projective real variety motivated by maximal ones is also maximal. 
\end{abstract}

\tableofcontents

\newpage
\section{Introduction}

\subsection{Background}
A \textit{real structure} on a complex manifold $X$ is an anti-holomorphic involution, that is, a diffeomorphism
\[\sigma \colon X\to X\]
satisfying $\sigma^2=\id_X$ and $\sigma^*I=-I$, where $I$ denotes the complex structure on $X$. In this paper, a \textit{real variety} (or $\mathbb{R}$-variety) refers to a pair $(X, \sigma)$ consisting of a complex manifold $X$ and a real structure $\sigma$ on it. If there is no risk of confusion, we may omit the reference to $\sigma$ in our notation.
The \textit{real locus} of $(X, \sigma)$, denoted by $X(\R)$, is defined to be the fixed locus of the involution $\sigma$. When $X(\R)\neq \emptyset$,  it is a differentiable submanifold of $X$, and its real dimension is equal to the complex dimension of $X$. These concepts can be extended to the setting of complex analytic spaces. 

In algebraic geometry, a \textit{real form} of a complex algebraic variety $X$ is an algebraic variety $X_0$ defined over $\R$ such that $X\simeq X_0\times_{\Spec \R} \Spec\C$ as $\C$-varieties. For a quasi-projective complex variety $X$, there is a natural bijection between  the set of its real forms (up to $\R$-isomorphism) and the set of real structures on $X(\C)$ (up to conjugation); see \cite[Exercise II.4.7]{Hartshorne}. In this paper, we will use both perspectives interchangeably. Except in Section \ref{sec:Motive}, $X$ denotes the complex space $X(\C)$.

As was included as the 16th problem in Hilbert's famous list \cite{Hilbert1900}, real algebraic geometry has a central theme of studying the topology of the real locus $X(\R)$ and its relation to the complex geometry of $X$. This subject predates complex algebraic geometry, as seen in the classical Harnack theorem from 1876 \cite{Harnack}, which states that a real plane curve of degree $d$ has at most $\frac{d^2-3d+4}{2}$ connected components.
In 1882, Klein \cite{Klein} found an intrinsic generalization of this theorem to any genus-$g$ compact Riemann surface equipped with a real structure, showing that the real locus can have at most $g+1$ connected components (which are circles). This result on curves is generalized to arbitrary dimensions through the \textit{Smith--Thom inequality}; see for example \cite[\S 1.2]{DIK-RealEnriquesSurfaces},  \cite[Theorem 3.3.6]{MangolteBook},

\begin{thm}[Smith--Thom inequality]
	\label{thm:SmithThom}
	Let $(X,\sigma)$ be a real variety of dimension $n$.
	We have the following inequality for the total $\F2$-Betti numbers 
	\begin{equation}
		\label{eqn:SmithThomInequality}
		b_*(X(\R), \F2)\leq b_*(X, \F2).
	\end{equation}
\end{thm}
Here  for a topological space $W$, its total $\F2$-Betti number is $b_*(W, \F2):=\sum_{i} b_i(W, \F2)$ with $b_i(W, \F2):=\dim_{\F2} H^i(W, \F2)$.

When equality in \eqref{eqn:SmithThomInequality} holds, we call $(X, \sigma)$ a \textit{maximal} real variety (or \textit{M-variety}).
Maximal real varieties have attracted significant research interest over the decades. 
Intriguing properties are uncovered for those real varieties, such as the Rokhlin congruence theorem for even-dimensional maximal smooth projective real varieties: $\chi_{\top}(X(\R))\equiv \operatorname{sgn}(X)  \mod 16$; see \cite[Theorem 3.4.2]{MangolteBook}. 
On the other hand, providing a rich supply of examples of maximal real varieties has always been an important task, but constructing them has proven challenging, particularly in higher dimensions. For curves and surfaces, the construction problem has been extensively studied, and for certain class of varieties (e.g.~K3 surfaces) even a classification up to real deformation is partially achieved; we refer to \cite{MangolteBook} for an account. However, the available  methods for constructing higher-dimensional maximal real varieties have been somewhat limited. Viro's combinatorial patchworking for hypersurfaces \cite{ItenbergViro} is currently the most powerful method. We refer to \cite[\S 3]{BrugalleSchaffhauser} for a recent summary.

\subsection{Results}
The main objective of this paper is to present a new type of constructions that produces new maximal varieties from existing ones. Our approach can be loosely described as ``taking moduli spaces of objects on a maximal variety'', where \textit{objects} can refer to algebraic cycles (or flags of such), vector bundles, coherent sheaves (or complexes of such), and so on. We have achieved the following concrete results.
\begin{enumerate}
	\item[(i)] We provide a short and computation-free  proof of Brugall\'e--Schaffhauser's result (Theorem \ref{thm:VBAC}) that for coprime integers $n>0$ and $d$, the moduli space $M_C(n, d)$ of stable vector bundles of rank $n$ and degree $d$ on a maximal curve $C$ is a maximal variety, and vice versa if $C(\R)\neq \emptyset$. The same result holds for the moduli spaces of stable parabolic bundles with full flag type (Corollary \ref{cor:Parabolic}).
	\item[(ii)]  By adapting our new proof  in (i) to the context of Higgs bundles, we establish the maximality of the moduli space $H_C(n, d)$ of stable Higgs bundles over a maximal curve $C$ of coprime rank and degree (Theorem \ref{thm:Higgs}). 
	\item[(iii)]  If a real variety $X$ has maximal Hilbert square $X^{[2]}$ and $X(\R)\neq \emptyset$, then $X$ itself, its Hilbert cube $X^{[3]}$, and the nested Hilbert schemes $X^{[1,2]}$ and $X^{[2,3]}$ are all maximal (Theorem \ref{thm:HilbertSquareCube}). 
	\item[(iv)] For a maximal real smooth cubic threefold, all the smooth (nested) Hilbert schemes are maximal (Theorem \ref{thm:Cubic3folds}). On the contrary, for a maximal real smooth cubic fourfold, its Hilbert square is never maximal (Theorem \ref{thm:NonMaxHilbCubic4}).
	\item[(v)]  If a smooth projective maximal  $\R$-surface $S$ satisfies $H^1(S, \F2)=0$  and $S(\mathbb{R})$ is connected, then the $n$-th punctual Hilbert scheme $S^{[n]}$ is maximal, for any $n>0$ (Theorem \ref{thm:HilbertPower}). This result applies, for example, to maximal $\R$-rational real surfaces (Corollary \ref{cor:RealRationalSurface}),  to some surfaces of Kodaira dimension 1 (Corollary \ref{cor:Dolgachev}), and to smooth projective real surfaces with Tate motives (Theorem \ref{thm:SurfaceTate}).
	\item[(vi)]  For integers $r>0, c_1,  c_2$ with $	\gcd\left(r, c_1, \frac{c_1(c_1+1)}{2}-c_2\right)=1$, the moduli space of stable sheaves on $\PP^2_{\R}$ with rank $r$ and first and second Chern classes $c_1$ and $c_2$ is also maximal (Theorem \ref{thm:ModuliP2}).
	\item[(vii)]  For a Poisson $\R$-surface $S$ satisfying the K-maximality condition (Definition \ref{def:KMax}) and a generic real ample line bundle $H$, the moduli spaces $M_H(S, v)$ of $H$-stable sheaves with Mukai vector $v$ on $S$ are maximal (Theorem \ref{thm:PoissonSurface}). This result applies, for example, to $\R$-rational maximal real Poisson surfaces (Corollary \ref{cor:RationalPoisson}).
\end{enumerate}
In all of our results, we implicitly include the claim that the moduli space (or Hilbert scheme) under consideration has a natural real structure. The maximality referred to in the conclusion is with respect to this real structure.

\subsection{Motives} 
In addition to providing examples of maximal varieties obtained via moduli constructions, another important aspect of the paper is to highlight, in Section \ref{sec:Motive}, the significant role played by \textit{motive} and its \textit{equivariant formality} in the study of maximal real varieties.

	We observe that the maximality of real algebraic varieties is a \textit{motivic} property, determined solely by the homological real motive with $\F2$-coefficients of the real variety (Definition \ref{def:MaximalMotive}). As a consequence (Corollary \ref{cor:MotivationMaximal}),\textit{ if a real algebraic variety $Y$ has motive contained in the tensor subcategory generated by the motives of some maximal varieties $\{X_i\} _{i}$ (in which case we say that $Y$ is \textit{motivated} by the $X_i$'s), then $Y$ is also maximal.}  From this motivic perspective, many results established in this paper can be explained or reproved. 
Further details can be found in  Section \ref{sec:Motive}.

\subsection{Branes in hyper-K\"ahler manifolds}
	On a hyper-K\"ahler manifold $X$, an anti-holomorphic involution $\sigma$ is referred to as an (ABA) or (AAB) brane involution\footnote{The distinction between (ABA) and (AAB) depends on the action of the involution on the holomorphic symplectic form: (ABA) if it is preserved and (AAB) if it is reversed.}, and its fixed locus, which is precisely the real locus $X(\R)$, is known as an (ABA) or (AAB) brane in $X$, as defined in \cite{BaragliaSchaposnik} (see also \cite[\S 2.3]{FrancoJardimMenet}). Regarding to the Smith--Thom inequality, $X(\mathbb{R})$ is called a \textit{maximal} brane if the real variety $(X, \sigma)$ is maximal. Our aforementioned result (ii) on Higgs bundles provides examples of maximal (ABA) or (AAB) branes in hyper-K\"ahler manifolds of arbitrarily high dimensions, which are non-compact but have a complete hyper-K\"ahler metric. Another example of maximal brane in a non-compact hyper-K\"ahler manifold is provided by the cotangent bundle of a maximal $\mathbb{R}$-variety. 
	
	However, the situation is more intriguing for \textit{compact} hyper-K\"ahler manifolds. On the one hand, in (complex) dimension 2, maximal real K3 surfaces and abelian surfaces exist and have been thoroughly studied; see \cite{Kharlamov-K3-76}, \cite[Chapters IV, VIII]{Silhol-Surface-LNM}. On the other hand, Kharlamov and R{\u{a}}sdeaconu \cite{Kharlamov-Rasdeaconu-HilbertSquare} recently discovered that the Hilbert squares of maximal real K3 surfaces and Fano varieties of lines of maximal cubic fourfolds are \textit{never} maximal (see Remark \ref{rmk:LossOfMaximality}). This leads to a challenging question: \textit{does there exist maximal (ABA) or (AAB) branes in compact hyper-K\"ahler manifolds of dimension $\geq 4$?}

\subsection{Structure of the paper}
\begin{itemize}
		\item Section \ref{sec:Preliminaries} provides a brief overview of equivariant cohomology, Atiyah's $\KR$-theory, and Chern classes.
		\item Section \ref{sec:Operations} covers fundamental operations that preserve maximality of real varieties, including products (Lemma \ref{lemma:Product}), symmetric powers (Theorem \ref{thm:Franz}), projective bundles or flag bundles (Proposition \ref{prop:ProjBun}), blow-ups (Proposition \ref{prop:Blowup}), flips and flops (Remark \ref{rmk:FlipFlop}), generically finite surjection with odd degree (Proposition \ref{prop:SurjectionOddDegree}), the Albanese variety and the Picard variety (Proposition \ref{prop:AlbMax}),
	the intermediate Jacobian of a maximal real cubic threefolds (Proposition \ref{prop:Cubic3}), etc.  These results are either documented in the literature or well-known to experts.
	\item Section \ref{sec:VariantsMaximality} introduces $c_1$-maximality and K-maximality, discusses their relation to maximality, and provides some examples.
	\item The main results are presented and proved in Sections \ref{sec:VBAC},  \ref{sec:Higgs}, \ref{sec:HilbertSquare}, \ref{sec:HilbertSchemes}, \ref{sec:ModuliP2}, \ref{sec:Poisson}; please refer to the table of contents.
	\item In Section \ref{sec:Motive}, basic notions of motives are recalled, the notion of equivariant formality is introduced. Maximality is reformulated in the context of motives, and  Section \ref{subsec:Applications} revisits the results from a motivic point of view and provide some more applications.
\end{itemize}

\noindent\textbf{Acknowledgment:} 
I would like to express my gratitude to Olivier Benoist, Erwan Brugall\'e, Victoria Hoskins, Simon Pepin Lehalleur, Burt Totaro, and Jean-Yves Welschinger for their interest and invaluable comments on this work. I also extend special thanks to Viatcheslav Kharlamov for his enthusiastic participation in numerous discussions and for generously sharing his recent findings  in \cite{Kharlamov-Rasdeaconu-HilbertSquare}. Furthermore, I am grateful for the conference ``Real Aspects of Geometry"  held at CIRM in 2022, which provided me with many sources of inspiration.

\section{Preliminaries}
\label{sec:Preliminaries}
Throughout the paper, we denote by $G:=\operatorname{Gal}(\mathbb{C}/\mathbb{R})$, which is a cyclic group of order 2.
\subsection{Equivariant cohomology}

Let $X$ be a complex variety or manifold. A real structure $\sigma$ on $X$ gives rise to an action of $G$ on $X$, where the nontrivial element of $G$ acts by $\sigma$, an anti-holomorphic involution. 

Let $BG$ be the classifying space of $G$, and $EG$ its universal cover. These spaces are unique up to homotopy equivalence. In the case where $G\simeq \mathbb{Z}/2\mathbb{Z}$, we can use the infinite-dimensional real projective space $\mathbb{R}\mathbb{P}^{\infty}$ (resp.\,infinite-dimensional sphere $\mathbb{S}^{\infty}$), endowed with the weak topology, as a model for $BG$ (resp.\,$EG$). 

Let $X_G:=X\times^G EG$ be the Borel construction, where $\times^{G}$ denotes the quotient by the (free) diagonal action of $G$.
Given a coefficient ring $F$ (such as $\mathbb{Z}$ or $\F2$), we denote by $$H^*_G(X, F):= H^*(X_G, F)$$ the equivariant cohomology ring with $F$-coefficients. 

We note that there is a natural fiber sequence $$X\hookrightarrow X_G\to BG,$$
which gives rise to the Leray--Serre spectral sequence:
\begin{equation}
	\label{eqn:LerarySerre}
	E_2^{p,q}=H^p(G, H^q(X, \F2))\Rightarrow H^{p+q}_G(X, \F2).
\end{equation}
Here, the left-hand side is the group cohomology and $H^q(X, \F2)$ is viewed as a $G$-module.

More generally, if we have a $G$-module $\mathcal{F}$, such as $\Z(r):=(\sqrt{-1})^r\Z$, which is the $G$-module with underlying group $\Z$ upon which the nontrivial element of $G$ acts by $(-1)^r$ (hence only the parity of $r$ matters), then we have the corresponding local system on the Borel construction $X_G$,  still denoted by $\mathcal{F}$. We define $H^*_G(X, \mathcal{F})$ to be the local cohomology $H^*(X_G, \mathcal{F})$, and the Leray--Serre spectral sequence \eqref{eqn:LerarySerre} generalizes to this setting. For example:
\begin{equation}
\label{eqn:LerarySerre-Coeff}
E_2^{p,q}=H^p(G, H^q(X, \mathbb{Z}(1)))\Rightarrow H^{p+q}_G(X, \mathbb{Z}(1)).
\end{equation}

\begin{rmk}[Algebraic setting]
For an algebraic variety $X_0$ defined over $\mathbb{R}$, and a torsion $G$-module $\mathcal{F}$, viewed as an \'etale sheaf on $\Spec(\mathbb{R})$, the equivariant cohomology $H^*_G(X_0(\mathbb{C}), \mathcal{F})$ is canonically identified with the \'etale cohomology $H^*_{\et}(X_0, \mathcal{F})$, where $\mathcal{F}$ also denotes its pullback to $X_0$ (see \cite[Corollary 15.3.1]{Sheiderer-LNM-Real&EtaleCohom}). Under this identification, the Leray--Serre spectral sequence is identified with the Hochschild--Serre spectral sequence.
\end{rmk}

We will use the following basic fact.

\begin{lemma}
	\label{lemma:Injectivity}
	Let $X$ be a real variety with $X(\mathbb{R})\neq \emptyset$ and $F$ a $G$-module (e.g. $\F2$, $\mathbb{Z}$, or $\mathbb{Z}(1)$). Then, for any $i$, the following map induced by the projection $\pi\colon X_G\to BG$ is injective:
	\[H^i(G, F)\xrightarrow{\pi^*}H^i_G(X, F).\]
\end{lemma}
\begin{proof}
	We can choose a real point $x\in X(\mathbb{R})$, which induces a section $s\colon BG\to X_G$ of $\pi$. Therefore, $s^*\circ \pi^*=\id$ and $\pi^*$ is injective. 
\end{proof}

\subsection{Maximality}
Recall that a real variety $(X, \sigma)$ is called \textit{maximal} if the Smith--Thom inequality in Theorem \ref{thm:SmithThom} becomes an equality: $b_*(X, \F2)=b_*(X(\mathbb{R}), \F2)$.

We will use repeatedly the following result, which reformulates the maximality of a real variety $(X, \sigma)$ in terms of the surjectivity of the restriction map from equivariant cohomology to ordinary cohomology. A proof in a more general setting can be found, for instance, in \cite[Chapter III, Proposition 4.16]{TomDieck}.
\begin{prop}[Maximality and equivariant formality]
	\label{prop:MaximalCriterion}
	Let $(X, \sigma)$ be a real variety. The following conditions are equivalent:
	\begin{enumerate}
		\item[(i)] $X$ is maximal.
		\item[(ii)]  The action of $G$ on $H^*(X, \F2)$ is trivial and the Leray--Serre spectral sequence \eqref{eqn:LerarySerre} degenerates at the $E_2$-page. 
		\item[(iii)]  The restriction homomorphism $H^*_G(X, \F2)\to H^*(X, \F2)$ is surjective.
	\end{enumerate}
\end{prop}
The  condition (iii) is called the \textit{equivariant formality} for the action of $G$ on $X$ (with coefficients $\F2$), see \cite{Franz}. This notion will be extended to motives in Section \ref{sec:Motive}.

We state the following simple fact for future reference. 
\begin{lemma}\label{lemma:Decomposition}
	Let $(X, \sigma)$ be a real variety, and let $i$ be a natural number. Assume that $H^i(X, \mathbb{Z})$ and $H^{i+1}(X, \mathbb{Z})$ are both 2-torsion-free. If the action of $G$ on $H^i(X, \F2)$ is trivial (e.g.~when $X$ is maximal), then we have
	\[H^i(X, \mathbb{Z})=H^i(X, \mathbb{Z})^G\oplus H^i(X, \mathbb{Z}(1))^G,\]
	where $H^i(X, \mathbb{Z}(1))^G$ denotes  the $\sigma$-anti-invariant part of $H^i(X, \mathbb{Z})$.
\end{lemma}
\begin{proof}
By the 2-torsion freeness assumption, $H^i(X, \mathbb{Z})^G\cap H^i(X, \mathbb{Z}(1))^G=0$ and we have an exact sequence
\[0\to H^i(X, \mathbb{Z})\xrightarrow{\times 2} H^i(X, \mathbb{Z})\to H^i(X, \F2)\to 0.\]
For any $\alpha\in H^i(X, \mathbb{Z})$, since $\sigma$ acts trivially on its image  $\overline{\alpha}\in H^i(X, \F2)$ by assumption, the above exact sequence implies that there exist unique elements $\alpha_1, \alpha_2\in H^i(X, \mathbb{Z})$ such that $\alpha+\sigma^*(\alpha)=2\alpha_1$ and $\alpha-\sigma^*(\alpha)=2\alpha_2$. Again by the 2-torsion freeness assumption, we have $\alpha_1\in H^i(X, \mathbb{Z})^G$, $\alpha_2\in H^i(X, \mathbb{Z}(1))^G$ and $\alpha=\alpha_1+\alpha_2$. The direct sum decomposition is proved.
\end{proof}

\subsection{K-theory}
\label{subsec:Ktheory} 
For a topological space $X$, Atiyah \cite{Atiyah-K-theory} introduced the topological (complex) K-theory, denoted by $\KU^*(X)$ in this paper\footnote{The notation $K^*_{\top}$ is also commonly used in the literature.}. One key property of the topological K-theory is the Bott periodicity, which states that $\KU^*(X)\simeq \KU^{*+2}(X)$.  Specifically, $\KU^0(X)$ is defined as the Grothendieck group of the category of $\mathcal{C}^\infty$ complex vector bundles on $X$. Chern classes can be defined on $\KU^0(X)$ using standard axioms, which allows the construction of the Chern polynomial map:
\[c_t\colon \KU^0(X)\to 1+\bigoplus_{r>0} H^{2r}(X, \Z)t^r.\]
This map sends $[E]$ to $1+c_1(E)t+c_2(E)t^2+\cdots$,  where $[E]$ denotes the class of a complex vector bundle $E$.

Now, let $(X, \sigma)$ be a real variety. A \textit{real vector bundle} on $X$, in the sense of Atiyah \cite{Atiyah-KR}, is a complex\footnote{in the $\mathcal{C}^\infty$ (topological) sense, not necessarily holomorphic or anti-holomorphic.} vector bundle $\pi\colon E\to X$ together with an involution $\tilde\sigma$ covering $\sigma$:
\begin{equation}
\xymatrix{
E \ar[r]^{\tilde\sigma} \ar[d]_{\pi}& E\ar[d]^{\pi}\\
X \ar[r]^{\sigma}& X
}
\end{equation}
such that the map $\tilde\sigma$ on the fibers $E_x\to E_{\sigma(x)}$ is $\C$-anti-linear for any $x\in X$.

\begin{rmk}\label{rmk:HolomorphicRealBundle}
To avoid confusion, it is important to note that, just like in the complex case, a real vector bundle $(E, \tilde{\sigma})$ does not necessarily have a holomorphic structure, and as a manifold, $E$ does not have a natural complex structure.  Therefore, it does not make sense to say that $\tilde{\sigma}$ is anti-holomorphic. 
However, in the holomorphic context, we can define a \textit{holomorphic real vector bundle} (also known as $\R$-bundle) on a real variety $(X, \sigma)$. This is a holomorphic vector bundle $E\to X$ together with a real structure (an anti-holomorphic involution) $\tilde\sigma$ on $E$, which covers $\sigma$ and induces a $\C$-anti-linear map $E_x\to E_{\sigma(x)}$ for any $x\in X$. According to GAGA, for a smooth projective  $\R$-variety with a corresponding real form $X_{\R}$, a holomorphic real vector bundle on $X_{\C}$ is an algebraic vector bundle on $X_{\R}$.
\end{rmk}

Atiyah \cite{Atiyah-KR} defined the $\KR$-theory for a real variety $(X, \sigma)$, denoted by $\KR^*(X)$, which is a graded commutative ring, with unit given by the class of the trivial line bundle $\C\times X$ with the complex conjugation on fibers. Specifically, $\KR^0(X)$ is the Grothendieck ring of the tensor category of real vector bundles on $X$. The $\KR$-theory $\KR^*(X)$ satisfies Bott periodicity $\KR^*(X)\simeq \KR^{*+8}(X)$ and interpolates between the (complex) topological K-theory, the equivariant K-theory,  and the $\operatorname{KO}$-theory. More details on $\KR$-theory can be found in the original source \cite{Atiyah-KR}.

Forgetting the involution on the bundle gives a natural map from the $\KR$-theory to the topological K-theory:
\begin{equation}
\KR^*(X)\to \KU^*(X).
\end{equation}

\subsection{Chern class}
\label{subsec:ChernClass}
For a real vector bundle $E$ on a real variety $X$ and  a positive integer $r\in \mathbb{N}$, the $r$th \textit{equivariant Chern class} of $E$,  denoted by $c^G_r(E)$,  is defined by Kahn \cite{Kahn-RealChernClass}. It is an element in $H^{2r}_G(X, \Z(r))$, where $\Z(r)=(\sqrt{-1})^r\Z$ is the Tate twist. We can define the equivariant Chern polynomial
$$c^G_t(E)\in 1+\bigoplus_{r>0} H^{2r}_G(X, \Z(r))t^r.$$
It is easy to see that equivariant Chern classes satisfy the same axioms as Chern classes. In particular, the additivity with respect to short exact sequences ensures that equivariant Chern classes are well-defined on $\KR^0(X)$.
For our purpose, we are mostly interested in cohomology with $\F2$-coefficients. We often view $c_r^G(E)$ as an element in $H^{2r}_G(X, \F2)$, by composing with the natural map $H^{2r}_G(X, \Z(r))\to H^{2r}_G(X, \F2)$, . (Note that the Tate twists are dropped when passing to $\F2$-coefficients since the $G$-module $\F2(r)$ is  always isomorphic to the trivial one $\F2$.)

We have the following commutative diagram:
\begin{equation}
	\xymatrix{
\KR^0(X) \ar[rr] \ar[d]^{c^G_t}& &\KU^0(X) \ar[d]^{c_t}\\
	1+\bigoplus_{r> 0} H^{2r}_G(X, \Z(r))t^r \ar[rr] \ar[d]&& 1+\bigoplus_{r>0} H^{2r}(X, \Z)t^r \ar[d]\\
		1+\bigoplus_{r>0} H^{2r}_G(X, \F2)t^r \ar[rr]&& 1+\bigoplus_{r>0} H^{2r}(X, \F2)t^r.
}
\end{equation}
where the top horizontal arrow is induced by forgetting the involution, the middle and bottom horizontal arrows are restriction maps.

On the level of cohomology, there is a natural map of ``restricting to the real locus'', which was introduced by Krasnov \cite{Krasnov-94} and van Hamel \cite{vanHamel-thesis}:
\[H^{2r}_G(X, \F2)\to H^r(X(\R), \F2).\]
For a real vector bundle $E$, we define $c_r^R(E)\in H^{r}_G(X(\R), \F2)$ as the image of the equivariant Chern class $c_r^G(E)\in H^{2r}_G(X, \F2)$ under the above map; this element is called the $r$th \textit{real Chern class} of $E$. In the algebraic setting, for an algebraic vector bundle $E$ on a real variety $X_{\R}$, we can equivalently define $c_r^R(E)$ as the image of the Chow-theoretic Chern class $c_r(E)\in \CH^r(X_{\R})$ under the Borel--Haefliger cycle class map \cite{BorelHaefliger} $\cl_R\colon \CH^r(X_{\R})\to H^r(X(\R), \F2)$.

\section{Basic operations preserving maximality}
\label{sec:Operations}
For facilitate easy reference, this section summarizes various operations that enable the construction of new maximal real varieties from existing ones. 

\subsection{Product} 
\begin{lemma}
	\label{lemma:Product}
Let $X_1, \dots, X_n$ be real varieties. Then the product $X:=X_1\times \dots\times X_n$ (endowed with the product real structure) is maximal if and only if each $X_i$ is maximal.
\end{lemma}
\begin{proof}
By the K\"unneth formula $H^*(X, \F2)\simeq \bigotimes_i H^*(X_i, \F2)$, we have 
\begin{equation}\label{eqn:KunnethComplex}
b_*(X, \F2)=\prod_i b_*(X_i, \F2).
\end{equation}

Since $X(\R)=X_1(\R)\times \cdots \times X_n(\R)$, the K\"unneth formula yields that 
\begin{equation}
	\label{eqn:KunnethReal}
b_*(X(\R), \F2)=\prod_i b_*(X_i(\R), \F2).
\end{equation}
Comparing \eqref{eqn:KunnethComplex} and \eqref{eqn:KunnethReal}, we see immediately that the maximality of all $X_i$'s implies the maximality of $X$, and also conversely by combining with the Smith--Thom inequality (Theorem \ref{thm:SmithThom}).
\end{proof}

\subsection{Symmetric power}
For a positive integer $n$, the $n$-th symmetric power of a complex variety $X$, denoted by $X^{(n)}$ or $\Sym^n(X)$, is defined as the quotient of $X^n$ by the permutation action of the symmetric group $\mathfrak{S}_n$. More generally, for any subgroup $\Gamma$ of $\mathfrak{S}_n$, one can define the $\Gamma$-product of $X$, denoted by $X^\Gamma$, as the quotient $X^n/\Gamma$.  A real structure $\sigma$ on $X$ induces a natural real structure $\sigma^{(n)}$ on $X^{(n)}$ by sending $\{x_1, \dots, x_n\}$ to $\{\sigma(x_1), \dots, \sigma(x_n)\}$, and more generally also a real structure $\sigma^\Gamma$ on $X^\Gamma$.

Note that even when $X$ is smooth, $X^\Gamma$ can be singular (for example, $X^{(n)}$ is singular when $\dim(X)\geq 2$ and $n\geq 2$), but the notion of real structures still makes sense. Nevertheless, the Smith--Thom inequality remains valid and we can speak of maximality. 

We record the following theorem of Franz \cite{Franz}, which generalizes an earlier work of Biswas and D'Mello on curves \cite{BiswasDMello}.
\begin{thm}
	\label{thm:Franz}
	If $(X, \sigma)$ is a maximal real variety, then for any $n>0$ and any subgroup $\Gamma$ of $\mathfrak{S}_n$, the $\Gamma$-product $(X^\Gamma, \sigma^\Gamma)$ is maximal. In particular, all the symmetric powers $(X^{(n)}, \sigma^{(n)})$ are maximal.
\end{thm}

\subsection{Projective bundle}
Let $X$ be complex manifold or complex algebraic variety, and $E$ a \textit{holomorphic} vector bundle of rank $r+1$ on $X$. The relative projectivization gives rise to a complex manifold $\PP(E):=\PP_X(E)$ which is a $\mathbb{C}\mathbb{P}^r$-bundle over $X$:
$$\pi\colon \PP(E)\to X,$$
called the \textit{projective bundle} associated to $E$.
If moreover $(X, \sigma)$ is a real variety and $(E, \tilde{\sigma})$ is a \textit{holomorphic} real vector bundle on it (Remark \ref{rmk:HolomorphicRealBundle}), then $\PP(E)$ inherits a natural real structure, sending $(x, [v])$ to $(\sigma(x), [\tilde\sigma(v)])$, which makes $\pi$ a morphism of real varieties. (In the algebraic setting, it follows from the simple observation that the projective bundle construction is defined over the base field.)

\begin{prop}
	\label{prop:ProjBun}
Let $E$ be a holomorphic real vector bundle on a real variety $X$. Then $X$ is maximal if and only if the projective bundle $\PP(E)$ is maximal.
\end{prop}
\begin{proof}
	Let $\mathcal{O}_{\pi}(1)$ be the relative Serre bundle. Denote by $\xi:=c_1^G(\mathcal{O}_\pi(1))\in H^2_G(\PP(E), \F2)$ its first equivariant Chern class as well as its image, the usual first Chern class, in $H^2(\PP(E), \F2)$, and denote by $\xi_R:=c_1^R(\mathcal{O}_\pi(1))\in H^1(\PP(E)(\R), \F2)$ its first real Chern class (see \S \ref{subsec:ChernClass}). 
	
	Let $r+1$ be the rank of $E$. By definition, the restriction of $\xi$ to a fiber $\mathbb{C}\mathbb{P}^r$ is a generator of its cohomology ring $H^*(\mathbb{C}\mathbb{P}^r,  \F2)=\F2[\xi]/(\xi^{r+1})$. By the Leray--Hirsch theorem, 
	\begin{equation}
		H^*(\PP(E), \F2)\cong H^*(X, \F2)\otimes H^*(\mathbb{C}\mathbb{P}^r, \F2).
	\end{equation}
	In particular, 
	\begin{equation}\label{eqn:BettiProjBun}
			b_*(\PP(E), \F2)=b_*(X, \F2)b_*(\mathbb{C}\mathbb{P}^r, \F2)=(r+1)b_*(X, \F2).
	\end{equation}
	On the other hand, it is clear that the real locus $\PP(E)(\R)$ is a $\mathbb{R}\mathbb{P}^r$-bundle over $X(\R)$. Again, by construction, the restriction of $\xi_R$ to a fiber is a generator of the cohomology ring $H^*(\mathbb{R}\mathbb{P}^r, \F2)=\F2[\xi_R]/(\xi_R^{r+1})$. By the Leray--Hirsch theorem,
	\begin{equation}
	H^*(\PP(E)(\R), \F2)\cong H^*(X(\R), \F2)\otimes H^*(\mathbb{R}\mathbb{P}^r, \F2).
	\end{equation}
	Consequently, 
	\begin{equation}\label{eqn:BettiProjBunReal}
	b_*(\PP(E)(\R), \F2)=b_*(X(\R), \F2)b_*(\PP^r_{\R}, \F2)=(r+1)b_*(X(\R), \F2).
	\end{equation}
	Comparing \eqref{eqn:BettiProjBun} and \eqref{eqn:BettiProjBunReal}, we see that $X$ is maximal if and only if $\PP(E)$ is maximal.
\end{proof}

\begin{rmk}
	\label{rmk:FlagBundles}
	Proposition \ref{prop:ProjBun} can be easily generalized to flag bundles (e.g.~Grassmannian bundles) associated to a holomorphic real vector bundle. The argument is similar: in the above proof, one replaces the usual (resp.~equivariant, real) first Chern class of the relative Serre bundle by the usual (resp.~equivariant, real) Chern classes of the universal/tautological bundles, and apply Leray--Hirsch in the same way. A more conceptual proof is via motives (see \S \ref{subsec:Applications}).
\end{rmk}

\subsection{Blow-up}
Let $X$ be a smooth real variety and $Y$ be a smooth real subvariety of codimension $c\geq 2$. Then the blow-up $\Bl_YX$, as well as the exceptional divisor $E\cong \PP(N_{Y/X})$, has a natural real structure, where $N_{Y/X}$ is the normal bundle. We have the commutative diagram of morphisms of real varieties.
\begin{equation}
\label{diag:Blowup}
	\xymatrix{
E \ar[d]_{p}\ar[r]^-j&\Bl_YX  \ar[d]^{\tau}	\\
Y\ar@{^(->}[r]^{i} & X
}
\end{equation}
\begin{prop}
	\label{prop:Blowup}
	If $X$ and $Y$ are both maximal, then $\Bl_YX$ is maximal.
\end{prop}
\begin{proof}
We have the following commutative diagram (the coefficients $\F2$ are omitted):
\begin{equation}
\label{diag:BlowupFormula}
	\xymatrix{
H^*_G(X)	\oplus H^{*-2}_G(Y)\oplus \cdots \oplus H^{*-2(c-1)}_G(Y) \ar[rr] \ar@{->>}[d]&& H^*_G(\Bl_YX)\ar[d]\\
H^*(X)	\oplus H^{*-2}(Y)\oplus \cdots \oplus H^{*-2(c-1)} (Y) \ar[rr]_-{\simeq}&& H^*(\Bl_YX)
}
\end{equation}
where all the vertical arrows are restriction maps, and both horizontal maps are given by
\[(\alpha, \beta_1, \dots, \beta_{c-1})\mapsto \tau^*(\alpha)+\sum_{k=1}^{c-1} j_*(p^*(\beta_k)\smile \xi^{k-1}),\]
where $\xi$ is $c^G_1(\mathcal{O}_{p}(1))\in H^2_G(E)$ in the top row and is $c_1(\mathcal{O}_{p}(1))\in H^2(E)$ in the bottom row; the commutativity of the diagram follows from the fact that all maps in \eqref{diag:Blowup} are real maps and $\mathcal{O}_p(1)$ is an $\R$-line bundle on $E$.

Now in the diagram \eqref{diag:BlowupFormula}, the bottom arrow is an isomorphism by the blow-up formula\footnote{In fact the top arrow is also an isomorphism: this is the equivariant version of the blow-up formula; but we do not need this.}, and the left vertical arrow is surjective by the maximalities of $X$ and $Y$ (Proposition \ref{prop:MaximalCriterion}). Therefore the right vertical arrow is surjective. By Proposition \ref{prop:MaximalCriterion}, $\Bl_YX$ is maximal.
\end{proof}

\begin{rmk}[Flips and flops]
	\label{rmk:FlipFlop}
	There are other types of basic birational transformations that preserve maximality. For instance, a standard flip\footnote{Also called ordinary flip or Atiyah flip.} (or flop) of a maximal real variety along a maximal real flipping/flopping center is again maximal (this can be deduced from the blow-up formula, or by using Jiang \cite[Corollary 3.8]{Jiang_ChowProjectivization}).  Similarly, the Mukai flop of a maximal real hyper-K\"ahler manifold along a real flopping center is maximal (use for example \cite[Theorem 3.6]{FHPL-rank2}).
\end{rmk}

We give an example of application. Given a smooth compact complex manifold $X$ and an integer $n$, the configuration space $F(X, n)$ is the open subset of $X^n$ parameterizing $n$ ordered \textit{distinct} points in $X$.   Fulton  and MacPherson constructed in \cite{FultonMacPherson-CompacitificationConf} a compactification $X[n]$ of $F(X, n)$, such that the complement is a simple normal crossing divisor with geometric description. The quickest way to define $X[n]$ is as the closure of the embedding $F(X,n)\subset X^n\times \prod_{I} \Bl_{\delta}(X^I)$, where $I$ runs through all subsets of $\{1, \cdots, n\}$ with $|I|\geq 2$, $X^I$ is the $|I|$-the power of $X$, and $\delta$ is the small(est) diagonal of $X^I$. Clearly, a real structure on $X$ induces a real structure on $X[n]$.

\begin{prop}\label{prop:Configuration}
If $X$ is a maximal real variety, then $X[n]$, equipped with the induced real structure, is also maximal.
\end{prop}
\begin{proof}
	We use the geometric inductive construction of Fulton--MacPherson's compactification. We refer the reader to \cite[\S 3]{FultonMacPherson-CompacitificationConf} for the details. Roughly speaking, starting from $X^n$, $X[n]$ is obtained as a sequence of $ 2^n-n-1$ blow-ups along centers that are themselves successive blow-ups of similar form. By repeatedly using Proposition \ref{prop:Blowup}, one shows by induction that all the varieties obtained as well as the blow-up centers are maximal.
\end{proof}

\subsection{Generically finite surjection of odd degree}
\begin{prop}
	\label{prop:SurjectionOddDegree}
	Let $f\colon X\to Y$ be a surjective proper real morphism between smooth real varieties. Assume that 
	$f$ admits a rational multi-section of \textit{odd} degree; for example, when $f$ is generically finite of \textit{odd} degree. If $X$ is maximal, then $Y$ is maximal.
\end{prop}
\begin{proof}
	We have the following commutative diagram
	\begin{equation}
		\xymatrix{
	H^*_G(X, \F2) \ar[r]^{f^G_*}\ar@{->>}[d]& H^*_G(Y, \F2) \ar[d]\\
	H^*(X, \F2) \ar@{->>}[r]^{f_*}& H^*(Y, \F2) 	
	}
	\end{equation}
	where the horizontal arrows are push-forwards and the vertical arrows are restriction maps. We claim that the bottom map is surjective. Indeed,  let $S$ be the closure of a rational multi-section of odd degree, then by projection formula we have that  $f_*(f^*(-)\smile [S])=\deg(S/Y)\cdot\id$, which is the identity since $\deg(S/Y)$ is odd and we are working with $\F2$-coefficients.
	Therefore $f_*$ is surjective.	
	The left vertical map is surjective by the maximality of $X$. Therefore the right vertical arrow is also surjective. By Proposition \ref{prop:MaximalCriterion}, it yields that $Y$ is maximal. 
\end{proof}

\begin{rmk}
	The statement does not hold for maps of even degree. As a counter-example, take a maximal real surface $S$ with non-maximal Hilbert square $S^{[2]}$  (such surface exists by \cite{Kharlamov-Rasdeaconu-HilbertSquare}, see Remark \ref{rmk:LossOfMaximality}), then the degree-2 quotient map $\Bl_{\Delta}(S\times S)\to S^{[2]}$ is from a maximal variety to a non-maximal one. 
\end{rmk}

\subsection{Albanese and Picard}
\label{sec:PicAlb}

Recall that for a projective complex manifold $X$, there are two naturally associated abelian varieties, namely,
its Picard variety 
$$\Pic^0(X):=\frac{H^1(X, \mathcal{O}_X)}{H^1(X, \Z)},$$
and 
its Albanese variety
$$\Alb(X):=\frac{H^0(X, \Omega_X^1)^\vee}{H_1(X, \Z)}.$$
A real structure on $X$ naturally induces real structures on $\Alb(X)$ and $\Pic^0(X)$ (see \cite[\S 2.7]{CilibertoPedrini-RealAV}), they are moreover \textit{proper} in the sense of \cite{CilibertoPedrini-RealAV}. The theory of real abelian varieties was developed by Comessatti \cite{Comessatti-RealAVI} \cite{Comessatti-RealAVII} ; see also \cite{CilibertoPedrini-RealAV} for a recent account.

\begin{prop}
	\label{prop:AlbMax}
	Let $X$ be a smooth projective real variety with $H_1(X, \Z)$ torsion-free. If $X$ is maximal,  then $\Alb(X)$ and $\Pic^0(X)$ are also maximal. In particular, the dual of a maximal real abelian variety is again maximal.
\end{prop}
\begin{proof}
	Recall that $\Alb(X)$ and $\Pic^0(X)$ are both of dimension $q:=h^{1,0}(X)$, the irregularity of $X$, and they both have total $\F2$-Betti number $2^{2q}$.
	Denote by $\sigma$ the real structure on $X$. Recall that the \textit{first Comessatti characteristic} is
	$$\lambda_1(X):=\dim_{\F2} (1+\sigma_*)H^1(X, \F2).$$
	The maximality of $X$ implies that $\sigma$ acts trivially on $H^1(X, \F2)$, hence $\lambda_1(X)=0$.
	
	We apply \cite[Theorem 2.7.2 and Theorem 2.7.9]{CilibertoPedrini-RealAV}, which say that $\Alb(X)(\R)$ and $\Pic^0(X)(\R)$ are both isomorphic as  Lie groups to 
	$(\mathbb{R}/\mathbb{Z})^q\times (\mathbb{Z}/2)^{q-\lambda_1}$. Since $\lambda_1=0$, the total $\F2$-Betti number is clearly $2^{2q}$, hence $\Alb(X)$ and $\Pic^0(X)$ are maximal. The assertion for dual abelian varieties follows since $\widehat{A}\cong \Pic^0(A)$, as real varieties.
\end{proof}

\begin{rmk}[Dual complex tori]
	In fact, the same argument shows that if a complex torus $T$ equipped  with a real structure is maximal, then its dual torus $\widehat{T}$, equipped with the natural real structure, is also maximal.
\end{rmk}

\begin{cor}[Jacobian]
	\label{cor:Jacobian}
	Let $C$ be a smooth projective real curve. If $C$ is maximal, then its Jacobian $J(C)$ is also maximal. Conversely, if $J(C)$ is maximal and $C(\R)\neq \emptyset$, then $C$ is also maximal. 
\end{cor}
\begin{proof}
The first claim is a special case of Proposition \ref{prop:AlbMax}, since the Jacobian of $C$ is identified with $\Alb(C)$ (and also with $\Pic^0(C)$ by the Abel--Jacobi theorem). Conversely, the Abel--Jacobi map $a: C\to J(C)$, which is a real morphism,  induces an isomorphism $a^*\colon H^1(J(C), \F2)\xrightarrow{\cong} H^1(C, \F2)$. In the following commutative diagram
\begin{equation*}
\xymatrix{
	H^1_G(J(C), \F2)  \ar@{->>}[d]\ar[r]^{a^*}&	H^1_G(C, \F2)\ar[d]\\
	H^1(J(C), \F2)\ar[r]^{a^*}_{\cong}  & 	H^1(C, \F2),
}
\end{equation*}
the left vertical map is surjective and the bottom map is an isomorphism. Therefore
the right vertical map is surjective. $H^j_G(C, \F2)\to H^j(C, \F2)$ is clearly surjective for $j=0$ and also for $j=2$ as $C(\R)\neq \emptyset$. Hence $C$ is maximal by Proposition \ref{prop:MaximalCriterion}.
\end{proof}

\begin{rmk}
	Corollary \ref{cor:Jacobian} can also be deduced from Franz's Theorem \ref{thm:Franz}: for $N\geq 2g+2$, the symmetric power $C^{(N)}$, which is maximal by  Theorem \ref{thm:Franz}, is a projective bundle over $J^N(C)$. By Proposition \ref{prop:ProjBun}, $J^{N}(C)$ is maximal. Since $C(\R)\neq \emptyset$, $J(C)$ is isomorphic to $J^N(C)$ as real varieties, hence is maximal.
\end{rmk}

\subsection{Intermediate Jacobians}
\label{sec:IntermediateJacobian}
For a smooth projective variety $X$ of odd dimension $n=2m+1$ with the only non-zero Hodge numbers $h^{i,i}(X)$ for all $0\leq i\leq n$ and $h^{m, m+1}=h^{m+1, m}$. Its intermediate Jacobian, denoted by $J(X)$, is a principally polarized abelian variety with the underlying complex torus 
\[J(X):=\frac{H^n(X, \C)}{F^{m+1}H^n(X, \C)\oplus H^n(X, \Z)_{\operatorname{tf}}}=\frac{H^{m, m+1}(X)}{H^n(X, \Z)_{\operatorname{tf}}},\]
and with polarization induced by the intersection pairing on $H^n(X, \Z)$.

Examples of such varieties include cubic threefolds, cubic fivefolds, quartic threefolds, quartic double solids, complete intersections of two even-dimensional quadrics, complete intersections of three odd-dimensional quadrics, odd-dimensional Gushel--Mukai varieties, etc.
A real structure  $\sigma$ on $X$ induces a $\C$-anti-linear map $H^{m, m+1}(X)\to H^{m+1, m}(X)=\overline{H^{m, m+1}(X)}$ respecting $H^n(X, \Z)_{\operatorname{tf}}$, and thus a natural real structure on $J(X)$. The following result for the Fano surface is due to Krasnov \cite{Krasnov-Cubic}.

\begin{prop}\label{prop:Cubic3}
	Let $X$ be a smooth real cubic threefold. If $X$ is maximal, then its Fano surface of lines $F(X)$ and the intermediate Jacobian $J(X)$ are also maximal.
\end{prop}
\begin{proof}
	We recall the proof for convenience of readers. 
	Let $\mathcal{P}:=\{(x, L)\in X\times F(X)~|~ x\in L\}$ be the incidence variety, or equivalently, the universal projective line over $F(X)$. 	
	The correspondence via $\mathcal{P}$ induces an isomorphism (\cite{ClemensGriffiths}):
	\begin{equation}
	\label{eqn:AJisom}
	[\mathcal{P}]^*\colon H^3(X, \Z)\xrightarrow{\cong} H^1(F(X), \Z).
	\end{equation}
	Consider the following commutative diagram:
	\begin{equation*}
	\xymatrix{
		H^3_G(X, \F2)\ar[r]^-{\cl_G(\mathcal{P})^*} \ar@{->>}[d]&H^1_G(F(X), \F2)\ar[d]\\
		H^3(X, \F2)\ar[r]_-{\cong}^-{[\mathcal{P}]^*} &H^1(F(X), \F2)
	}
	\end{equation*}
	where the top map is the correspondence induced by the equivariant cycle class $\cl_G(\mathcal{P})$, the bottom map is an isomorphism by \eqref{eqn:AJisom} and the universal coefficient theorem,  the vertical arrows are restriction maps and the left one is surjective by assumption. Therefore the right arrow $H^1_G(F(X), \F2)\to H^1(F(X), \F2)$ is surjective too. Since $H^2(F(X), \F2)=\wedge^2H^1(F(X)), \F2)$, we have that $H^*_G(F(X), \F2)\to H^*(F(X), \F2)$  is surjective, hence $F(X)$ is maximal. 
	Consequently, by Proposition \ref{prop:AlbMax}, $\Alb(F(X))$ is also maximal.
	Finally, by \cite{ClemensGriffiths}, we have $J(X)\cong \Alb(F(X))$ and it is straightforward to check that the isomorphism is defined over $\R$. Therefore, $J(X)$ is also maximal.
\end{proof}
The same arguments show similar results as follows:
\begin{itemize}
	\item for a general real cubic fivefold $X$ that is maximal,  its Fano variety of planes $F_2(X)$ and its intermediate Jacobian $J(X)$ are also maximal;
	\item for a general real quartic threefold $X$ that is maximal, its Fano variety of conics $F_c(X)$ and  its intermediate Jacobian $J(X)$ are also maximal;
	\item for a Gushel--Mukai  3-fold or 5-fold $X$ that is maximal, its double EPW surface $\widetilde{Y}^{\geq 2}_{A(X)}$ and  its intermediate Jacobian $J(X)$ are also maximal;
\end{itemize}
where in place of Clemens--Griffiths \cite{ClemensGriffiths}, one applies results of Collino \cite{Collino}, Letizia \cite{Letizia}, Debarre--Kuznetsov \cite{DebarreKuznetsov-IJ} respectively.

\section{Variants of maximality}
\label{sec:VariantsMaximality}
Inspired by Proposition \ref{prop:MaximalCriterion}, we propose several variants of the notion of maximality for real varieties.
\begin{defn}[K-maximality]
	\label{def:KMax}
	Let $X$ be a real variety. We say that $X$ is $K^0$-\textit{maximal}, if the natural map forgetting the real structure
	\[ \KR^0(X)\to \KU^0(X) \]
	is surjective. Similarly, one can define $K^1$-maximality by the surjectivity of $\KR^1(X)\to \KU^1(X)$.\\
	We say that $X$ is 	\textit{K-maximal} if it is $K^0$-maximal and $K^1$-maximal.
\end{defn}

\begin{defn}[$c_1$-maximality]
	\label{def:c1Max}
	Let $X$ be a real variety. We say that $X$ is \textit{$c_1$-maximal}, if the first Chern class map 
	\[ c_1\colon \KR^0(X)\to H^2(X, \F2) \]
	is surjective. 
	By first taking the determinant bundle, it is clear that the $c_1$-maximality is equivalent to the surjectivity of 
	\[c_1\colon \{\text{real vector bundles of rank 1} \}\to H^2(X, \F2).\]
	Thanks to \cite[Proposition 1]{Kahn-RealChernClass}, this condition is also equivalent to the surjectivity of the following natural map
	\[H^2_G(X, \Z(1)) \to H^2(X, \F2).\]
	
\end{defn}

It would be very interesting to compare these notions to the usual maximality. First, we have a few easy relations:

\begin{lemma}
	\label{lemma:RelationMaximalities}
	Let $X$ be a smooth projective real variety. 
	\begin{enumerate}
		\item[(i)] If $X$ is $K^0$-maximal \footnote{The argument in the proof shows that the $K^0$-maximal condition can be weakened to the surjectivity of $\KR^0(X)\otimes \mathbb{Z}/2\mathbb{Z}\to \KU^0(X)\otimes \mathbb{Z}/2\mathbb{Z}$.} and $H^3(X, \Z)$ has no 2-torsion, then $X$ is $c_1$-maximal.
		\item[(ii)] If $X$ is a $c_1$-maximal $\mathbb{R}$-surface such that $X(\mathbb{R})\neq \emptyset$ and $b_1(X)=0$, then $X$ is maximal.
	\end{enumerate}
\end{lemma}
\begin{proof}
	For (i), the map $c_1\colon \KR^0(X)\to H^2(X, \F2) $ is the composition:
	\[\KR^0(X)\longrightarrow \KU^0(X)\xrightarrow{c_1} H^2(X, \Z) \to H^2(X, \F2).\]
	The first map is surjective by assumption, the second map is surjective since $K(\Z, 2)\cong \mathbb{C}\mathbb{P}^{\infty}$ classifies vector bundles of rank 1, the third map is surjective since $H^3(X, \Z)$ has no 2-torsion. 
	
	For (ii),  since $c_1 \colon \KR^0(X)\to H^2(X, \F2)$ factors through the equivariant first Chern class map $c_1^G \colon \KR^0(X)\to H^2_G(X, \F2)$, we get that $H^2_G(X, \F2)\to H^2(X, \F2)$ is surjective. Since $b_1(X)=0$, $H^j_G(X, \F2)\to H^j(X, \F2)$  is surjective for $j=1, 3$. Since $X(\mathbb{R})\neq \emptyset$, $H^4_G(X, \F2)\to H^4(X, \F2)$ is also surjective. Therefore, $X$ is maximal.
\end{proof}

For maximal real surfaces, the $c_1$-maximality can be reinterpreted geometrically as follows; the equivalence between $(ii)$ and $(iii)$ below is due to Kharlamov, who kindly taught me the nice geometric argument and allowed me to include it here.
\begin{prop}
	\label{prop:MaximalAndC1Maximal}
	Let $X$ be a maximal smooth projective real surface with $H^1(X, \F2)=0$.
Then the following conditions are equivalent:
\begin{enumerate}
	\item[(i)] $X$ is $c_1$-maximal, i.e.~$H^2_G(X, \Z(1)) \to H^2(X, \F2)$ is surjective;
	\item[(ii)] $\rk H^2(X, \mathbb{Z})^G=0$;
	\item[(iii)]$X(\mathbb{R})$ is connected.
\end{enumerate}
\end{prop}
\begin{proof}
	The condition $H^1(X, \F2)=0$ is equivalent to saying that $b_1(X)=b_3(X)=0$ and $H^*(X, \Z)$ is 2-torsion-free.\\
	For $(i) \Leftrightarrow (ii)$,
	by the Leray--Serre spectral sequence \eqref{eqn:LerarySerre-Coeff}, together with the fact that $H^1(X, \mathbb{Z})=0$ (since $b_1=0$), we have an exact sequence
	\[0\to H^2(G, \mathbb{Z}(1))\to H^2_G(X, \mathbb{Z}(1))\to H^2(X, \mathbb{Z}(1))^G\to H^3(G, \mathbb{Z}(1))\to H^3_G(X, \mathbb{Z}(1)).\]
	Since $H^2(G, \mathbb{Z}(1))=0$ and $H^3(G, \mathbb{Z}(1))\to H^3_G(X, \mathbb{Z}(1))$ is injective by Lemma \ref{lemma:Injectivity} ($X(\R)\neq \emptyset$ since $X$ is maximal), the natural map 
	$H^2_G(X, \mathbb{Z}(1))\to H^2(X, \mathbb{Z}(1))^G$ is an isomorphism. Hence the natural map $H^2_G(X, \mathbb{Z}(1))\to H^2(X, \F2)$ factorizes as follows 
	\begin{equation}
	\label{eqn:Factorization}
	H^2_G(X, \mathbb{Z}(1))\xrightarrow{\simeq} H^2(X, \mathbb{Z}(1))^G\hookrightarrow H^2(X, \mathbb{Z})\to H^2(X, \F2).
	\end{equation}
	By Lemma \ref{lemma:Decomposition}, $$H^2(X, \mathbb{Z})=H^2(X, \mathbb{Z})^G\oplus H^2(X, \mathbb{Z}(1))^G.$$ Since $H^2(X, \F2)=H^2(X, \mathbb{Z})\otimes \F2$, the surjectivity of \eqref{eqn:Factorization} (i.e.~the $c_1$-maximality) is equivalent to the vanishing of $H^2(X, \mathbb{Z})^G\otimes \F2$, or equivalently, since $H^2(X, \mathbb{Z})^G$ is 2-torsion-free, $H^2(X, \mathbb{Z})^G$ is of rank zero.
	
	For $(ii)\Rightarrow (iii)$, supposing the contrary that $X(\mathbb{R})$ is disconnected, one can take two points $x, y$ belonging to distinct components of $X(\mathbb{R})$. Let $\gamma$ be a path from $x$ to $y$, then $\gamma-\sigma_*(\gamma)$ is a 1-cycle in $X$. Since $H_1(X, \mathbb{Z})$ is of odd torsion, there exists an odd number $N$ such that $N(\gamma-\sigma_*(\gamma))$ is a boundary of a 2-chain $T$, i.e.
	\[N(\gamma-\sigma_*(\gamma))=\partial T.\]
	Then $T+\sigma_*(T)$ is a 2-cycle, $\sigma$-invariant. Moreover, its class in $H_2(X, \mathbb{Z})$ is non-torsion, since the intersection number of $T+\sigma_*(T)$ with the connected component of $S(\mathbb{R})$ containing $x$ is odd. We thus obtain a free class in $H_2(X, \mathbb{Z})^G$, a contradiction.
	
	For $(iii)\Rightarrow (ii)$, assuming $X(\mathbb{R})$ is connected, then the maximality of $X$ implies that 
	\[\dim H^1(X(\mathbb{R}), \F2)= \rk  H^2(X, \Z)=\rk  H^2(X, \Z)^G+\rk  H^2(X, \Z(1))^G.\]
	Consider the Viro homomorphism (see \cite[Section 1.5]{DIK-RealEnriquesSurfaces}), 
	$$\operatorname{bv}\colon H_2(X, \mathbb{Z}(1))^G\to H_1(X(\mathbb{R}), \F2),$$
	which is, roughly speaking, defined as follows: for a 2-dimensional $\sigma$-invariant oriented real submanifold $\Sigma\subset X$ such that $\sigma$ reverses its orientation,  $\operatorname{bv}$ sends the class $[\Sigma]\in H_2(X, \mathbb{Z}(1))^G$ to the class of the smooth curve $\Sigma\cap X(\mathbb{R})$ in $H_2(X(\mathbb{R}), \F2)$.
	
	We claim that the Viro homomorphism $\operatorname{bv}$ is surjective. Indeed, for any 1-cycle $\gamma$ representing a class $H_1(X(\mathbb{R}), \F2)$, its image inside $H_1(X, \Z)$ is of odd torsion. Hence there is an odd number $N$, such that the class of $N\gamma$ vanishes, i.e., there exists a 2-chain $T$ in $X$, such that $\partial T=N\gamma$. Then $T-\sigma_*(T)$ is a 2-cycle, $\sigma$-anti-invariant, hence defines an element in $H^2(X, \Z(1))^G$, which is a preimage of $[\gamma]$.
	The surjectivity of the Viro homomorphism implies that $\rk  H^2(X, \Z(1))^G\geq \dim H^1(X(\mathbb{R}), \F2)$. Therefore, we must have $\rk  H^2(X, \Z)^G=0$.
\end{proof}

\begin{lemma}
	\label{lemma:KMaxRationalSurface}
	A maximal $\R$-rational real surface is K-maximal and $c_1$-maximal.
\end{lemma}
\begin{proof}
	The minimal model theory for surfaces has its real version: any smooth projective real surface can be obtained from a minimal\footnote{Here, a real surface is called minimal if it does not contain a (-1)-curve defined over $\R$ or a disjoint conjugate pair of (-1)-curves.} one  by a sequence of successive blow-ups, where at each step  the blow-up center is either a point in the real locus or a pair of complex conjugate points outside of the real locus; see \cite{Kollar-TopologyOfRealVar}.
	
	Note that given a maximal surface $S$, (any of) its minimal model $S_0$  is again maximal by Proposition \ref{prop:SurjectionOddDegree}. Moreover, at each step of blow-up in $S=S_n\to \cdots \to S_1\to S_0$, the blow-up center must be a point in the real locus (otherwise $b_2(S_{i+1}, \F2)-b_2(S_i, \F2)=2$, but the real locus does not change). Therefore, any maximal real surface can obtained from a maximal minimal real surface, via a sequence of blow-ups at points in the real loci. Since blowing up at a point in the real locus preserves the K-maximality, we only need to check that for any minimal rational real surface, maximality implies K-maximality (which implies $c_1$-maximality by Lemma \ref{lemma:RelationMaximalities}).
	
	By classification (see \cite[Theorem 4.3.23 and Theorem 4.4.11]{MangolteBook}), minimal rational real surfaces that are maximal are $\PP^2$, $\PP^1\times_{\R} \PP^1$, $\mathbb{F}_n$ (real Hirzebruch surface) for $n\geq 2$. For each surface $S$ in the list, $\KU^1(S)=0$ and $\KU^0(S)$ is generated by real algebraic line bundles: \\
	- for $\PP^2$, the generators are $\mathcal{O}, \mathcal{O}(1), \mathcal{O}(2)$;\\
	- for $\PP^1\times \PP^1$, the generators are $\mathcal{O},p_1^* \mathcal{O}(1), p_2^* \mathcal{O}(1), p_1^* \mathcal{O}(1)\otimes p_2^* \mathcal{O}(1)$;\\
	- for $\mathbb{F}_n=\PP_{\PP^1}(\mathcal{O}\oplus \mathcal{O}(n))\xrightarrow{\pi} \PP^1$, the generators are $\mathcal{O}, \pi^*\mathcal{O}_{\PP^1}(1), \mathcal{O}_\pi(1), \pi^*\mathcal{O}_{\PP^1}(1)\otimes \mathcal{O}_{\pi}(1)$;\\
	Hence $\KR^0(S)\to \KU^0(S)$ is surjective.
\end{proof}

\section{Moduli spaces of vector bundles on curves}
\label{sec:VBAC}
Let $C$ be a compact Riemann surface of genus $g\geq 2$. Let $n>0$ and $d$ be two coprime integers. Let $M:=M_C(n, d)$ be the moduli space of stable vector bundles of rank $n$ and degree $d$. Since $n$ and $d$ are assumed to be coprime, $M$ is a projective complex manifold of dimension $n^2(g-1)+1$ and there exists a universal bundle $\mathcal{E}$ on $C\times M$. We denote by $p\colon C\times M\to C$ and $\pi\colon C\times M\to M$ the two natural projections.

We will use the following fundamental theorem of Atiyah--Bott \cite[Theorem 9.11]{AtiyahBott} on generators of the cohomology ring of $M$.

\begin{thm}[Atiyah--Bott]
	\label{thm:AtiyahBott}
	Assume that $d>(2g-2)n$. 
	The cohomology ring $H^*(M, \Z)$ is torsion-free and generated by the following three types of elements.
	\begin{enumerate}
		\item[(i)] For $1\leq r\leq n$, the Chern class $a_r:=c_r(\mathcal{E}|_M)\in H^{2r}(M, \Z)$, where $\mathcal{E}|_M$ is the restriction of $\mathcal{E}$ to a fiber of $p$.
		\item[(ii)] For $1\leq r\leq n$, the $(1, 2r-1)$-K\"unneth components $b_{r,j}\in H^{2r-1}(M, \Z)$ of $c_r(\mathcal{E})$, where we write $c_r(\mathcal{E})=\sum_{j=1}^{2g} \alpha_j\otimes b_{r, j}+1_C\otimes a_r+ [\pt]\otimes \pi_*(c_r(\mathcal{E}))$, where $\{\alpha_j\}_{1\leq j\leq 2g}$ form a basis of $H^1(C, \Z)$, $1_C$ is the fundamental class of $C$ and $[\pt]$ is the class of a point on $C$.
		\item[(iii)] For $1\leq r\leq d+n(1-g)$, the Chern class $c_r(\pi_*\mathcal{E})\in H^{2r}(M, \Z)$.
	\end{enumerate}
\end{thm}

\begin{rmk}
	\label{rmk:Pushforward}
	In the above theorem, note that under the numerical assumption $d>(2g-2)n$, we have $R^1\pi_*\mathcal{E}=0$ and $\pi_*\mathcal{E}$ is a vector bundle of rank $d+n(1-g)$ by Riemann--Roch. This assumption can be dropped in the statement provided that we replace in (iii) the direct image $\pi_*\mathcal{E}$ by the K-theoretic push-forward $\pi_!\mathcal{E}:=\sum_i (-1)^i[R^i\pi_*\mathcal{E}]$ and allow $r$ to be any integer.
\end{rmk}

Now, let $\sigma$ be a real structure on the curve $C$. It naturally induces a real structure $\sigma_M$ on $M$, given by the anti-holomorphic involution sending $E$ to $\overline{\sigma^*(E)}$ (indeed, the rank and the degree are both preserved by $\sigma_M$). We give a short proof of the following theorem of Brugall\'e and Schaffhauser \cite[Theorem 1.2]{BrugalleSchaffhauser}.

\begin{thm}
	\label{thm:VBAC}
	Let $(C, \sigma)$ be a smooth projective real curve. If $(C, \sigma)$ is maximal, then $(M, \sigma_M)$ is also maximal. Conversely, if $(M, \sigma_M)$ is maximal and $C(\R)\neq \emptyset$, then $(C, \sigma)$ is maximal.
\end{thm}
\begin{proof}
	For simplicity we drop the real structures from the notation. First, we assume that $C$ is maximal and prove the maximality of $M$.
	Tensoring with an algebraic $\mathbb{R}$-line bundle $\mathcal{L}$ on $C$ of positive degree  induces an isomorphism of moduli spaces that is compatible with the real structures (in particular the maximality is preserved):
	\begin{equation*}
	-\otimes \mathcal{L}\colon M_C(n, d)\xrightarrow{\simeq} M_C(n, d+n\deg(\mathcal{L})).
	\end{equation*}
	We can therefore assume that $d>(2g-2)n$.
	
	By passing to $\F2$-coefficient in Theorem \ref{thm:AtiyahBott}, we get a set of generators for the ring $H^*_G(M, \F2)$. We will show that they are all in the image of the restriction map $H^*_G(M, \F2)\to H^*(M, \F2)$. The key point is that the universal bundle $\mathcal{E}$ is a real algebraic/holomorphic vector bundle on the real variety $C\times M$, hence defines a class $[\mathcal{E}]\in \KR^0(C\times M)$ and possesses equivariant Chern classes $c^G_r(\mathcal{E})\in H^{2r}_G(C\times M, \F2)$ for any $r\in \mathbb{N}$.
	
	- For the generators in $(i)$ and $(ii)$, they are all of the form $[c_r(\mathcal{E})]_*(\alpha):=\pi_*( p^*(\alpha)\smile c_r(\mathcal{E}))$, where $\alpha\in H^*(C, \F2)$, upon which the Chern class $c_r(\mathcal{E})\in H^*(C\times M, \F2)$ acts as cohomological correspondence. Since we have the following commutative diagram:
	\begin{equation}
	\xymatrix{
		H_G^*(C, \F2) \ar[rr]^{[c^G_r(\mathcal{E})]_*} \ar@{->>}[d]& &	H_G^*(M, \F2) \ar[d]\\
		H^*(C, \F2)  \ar[rr]^{[c_r(\mathcal{E})]_*} & &	H^*(M, \F2),
	}
	\end{equation}
	where the left vertical arrow is surjective by the maximality of $C$ (and Proposition \ref{prop:MaximalCriterion}), we obtain that $[c_r(\mathcal{E})]_*(\alpha)$ lies in the image of $H^*_G(M, \F2)\to H^*(M, \F2)$.
	
	- For the generators in $(iii)$, consider the following commutative diagram (for any $r\in \mathbb{N}$):
	\begin{equation}
	\xymatrix{
		\KR^0(C\times M) \ar[r]^-{\pi_!} \ar[d]& \KR^0(M)\ar[r]^{c^G_r} \ar[d]& H^{2r}_G(M, \F2)\ar[d]\\	
		\KU^0(C\times M) \ar[r]^-{\pi_!}& \KU^0(M)\ar[r]^{c_r} & H^{2r}(M, \F2),
	}
	\end{equation}
	where the left horizontal arrows are the K-theoretic push-forwards along the projection $\pi\colon C\times M\to M$ (which is a real map). Recall that $[\mathcal{E}]\in \KR^0(C\times M)$ is mapped to $[\mathcal{E}]\in \KU^0(C\times M)$. Therefore, $c_r(\pi_*(\mathcal{E}))=c_r(\pi_![\mathcal{E}])$ (see Remark \ref{rmk:Pushforward}), which naturally lies in the image of $H^*_G(M, \F2)\to H^*(M, \F2)$.
	
	Consequently, $H^*_G(M, \F2)\to H^*(M, \F2)$ is surjective. We conclude the maximality of $M$ by Proposition \ref{prop:MaximalCriterion}.
	
	Conversely, assuming that $M$ is maximal,  then by Proposition \ref{prop:AlbMax} and the fact that $H^*(M, \Z)$ is torsion-free, we get the maximality of the Albanese variety of $M$, which is identified with $J^d(C)$. Since $C(\R)\neq \emptyset$ by hypothesis, tensoring with a degree-1 real line bundle shows that $J(C)$ is isomorphic to $J^d(C)$, hence is also maximal. By Corollary \ref{cor:Jacobian}, $C$ is maximal. 
\end{proof}

\begin{rmk}
	\begin{enumerate}
		\item Brugall\'e--Schaffhauser \cite{BrugalleSchaffhauser} also proved that for a $\R$-line bundle $\mathcal{L}$ on a maximal  real curve $C$, the moduli space $M_C(n, \mathcal{L})$ of stable bundles with fixed determinant is also maximal.  
		\item Brugall\'e--Schaffhauser \cite{BrugalleSchaffhauser} actually shows that $M_C(n, d)$ and $M_C(n, \mathcal{L})$ satisfy a stronger property called \textit{Hodge-expressivity} provided $C$ is maximal. Our proof does not give any information on this regard; the reason is partially explained in Remark \ref{rmk:SummandHE}.		
	\end{enumerate}
\end{rmk}

Moduli spaces of parabolic vector bundles were introduced and studied by Mehta and Seshadri \cite{MS}. Theorem \ref{thm:VBAC} admits the following consequence in this context. 
\begin{cor}[Moduli of parabolic bundles]
	\label{cor:Parabolic}
	Let $C$ be a maximal real curve. Fix a (finite) set $D\subset  C(\R)$, whose elements are called parabolic points. Let $n>0$, $d$ be coprime integers. Let $\underline{1}$ denote the full flag type (at each point of $D$), and let $\alpha$ be a generic weight. Then the moduli space $M^{\alpha}_{C, D}(n, d, \underline{1})$ of $\alpha$-stable parabolic vector bundles of rank $n$, degree $d$ and full flag type is maximal. 
\end{cor}
\begin{proof}
	By \cite{BH}, \cite{BY_rationality}, \cite{Thaddeus_VGIT}, there is a wall-and-chamber structure on the space of weights, and by varying the weight, the moduli space stays the same within the chamber and undergoes a standard flip when crossing a wall with flipping centers being projective bundles over products of moduli spaces of stable parabolic bundles with smaller invariants (see \cite[Theorem 4.1]{BY_rationality} or \cite[Theorem 5.11]{FHPL-rank2}). Therefore, by wall-crossing, induction, and Remark \ref{rmk:FlipFlop}, we only need to prove the maximality of $M^{\alpha}_{C, D}(n, d, \underline{1})$ for $\alpha$ generic in the minimal chamber, where the stability of a parabolic bundle reduces to the stability of the underlying vector bundle. To this end, one uses natural morphisms which forget the flags. By \cite[Theorem 4.2]{BY_rationality}, such morphisms identify the moduli space of parabolic bundles as iterated flag bundles (hence iterated Grassmannian bundles) over $M_C(n,d)$. By Proposition \ref{prop:ProjBun} (or rather Remark \ref{rmk:FlagBundles}) and Theorem \ref{thm:VBAC}, we can conclude the maximality of $M^{\alpha}_{C, D}(n, d, \underline{1})$.
\end{proof}

\section{Moduli spaces of Higgs bundles on curves}
\label{sec:Higgs}
Let $C$ be smooth projective curve of genus $g\geq 2$. A \textit{Higgs bundle} on $C$ is a pair $(E, \phi)$ where $E$ is a holomorphic vector bundle on $C$ and $\phi\colon E\to E\otimes \omega_C$ is $\mathcal{O}_C$-linear. Such a pair $(E, \phi)$ is called (semi-)stable if  any proper Higgs sub-bundle $(E', \phi|_{E'})$ satisfies the slope inequality $\mu(E')(\leq)<\mu(E)$.  Fix two integers $n>0$ and $d$ that are coprime to each other. Then stability and semi-stability coincide for Higgs bundles of rank $n$ and degree $d$, and the moduli space $H_C(n, d)$, constructed in \cite{Nitsure}, of (semi-)stable Higgs bundles of rank $n$ and degree $d$ is a smooth quasi-projective variety of dimension $2n^2(g-1)+2$, admitting a hyper-K\"ahler metric and containing the cotangent bundle of $M_C(n, d)$ as an open subscheme. There is the so-called Hitchin map sending $(E, \phi)$ to the coefficients of the characteristic polynomial of $\phi$:
$$H_C(n,d)\to \bigoplus_{i=1}^nH^0(C, \omega_C^{\otimes i}).$$ 
The Hitchin map is a proper map.  Hence $H_C(n,d)$ is a partial compactification\footnote{in the sense that it compactifies the fibers of the Hitchin map by adding semistable Higgs bundles with unstable underlying vector bundle.} of the cotangent bundle of $M_C(n, d)$. We refer to \cite{Hitchin-Selfduality}, \cite{Simpson-IHES-Higgs} for details.

Markman constructed in \cite[Theorem 3]{Markman-IntegralGeneratorPoisson} a system of integral generators of the cohomology ring $H^*(H_C(n,d), \Z)$. By \cite[Corollary 9]{GarciaPradaHeinloth-Duke}, $H^*(H_C(n,d), \Z)$ is torsion-free, hence Markman's generators also give generators for $H^*(H_C(n,d), \F2)$. We will need the following somewhat different characterizations of these integral generators, in a way closer to that in Atiyah--Bott's Theorem \ref{thm:AtiyahBott}.

\begin{thm}[Integral generators for cohomology of Higgs moduli]
	\label{thm:GeneratorHiggs}
	Let $H:=H_C(n, d)$ be the moduli space of stable Higgs bundles of coprime rank $n$ and degree $d$. Let $\mathcal{E}$ be a universal bundle over $C \times H$. Let $p_1, p_2$ be the two projections from $C\times H$ to its two factors. Then the cohomology ring $H^*(H, \Z)$ is torsion-free and generated by the following three types of elements.
	\begin{enumerate}
		\item[(i)] For $1\leq r\leq n$, the Chern class $a_r:=c_r(\mathcal{E}|_H)\in H^{2r}(H, \Z)$, where $\mathcal{E}|_H$ is the restriction of $\mathcal{E}$ to a fiber of $p_1$.
		\item[(ii)] For $1\leq r\leq n$, the $(1, 2r-1)$-K\"unneth components $b_{r,j}\in H^{2r-1}(H, \Z)$ of $c_r(\mathcal{E})$, where we write $c_r(\mathcal{E})=\sum_{j=1}^{2g} \alpha_j\otimes b_{r, j}+1_C\otimes a_r+ [\pt]\otimes p_{2,*}(c_r(\mathcal{E}))$ via the K\"unneth formula, where $\{\alpha_j\}_{1\leq j\leq 2g}$ form a basis of $H^1(C, \Z)$, $1_C$ is the fundamental class of $C$ and $[\pt]$ is the class of a point on $C$.
		\item[(iii)] For $r\geq 0$, the Chern class $c_r(p_{2, !}[\mathcal{E}])\in H^{2r}(H, \Z)$, where $[\mathcal{E}]$ is viewed as in $\KU^0(C\times H)$ and $p_{2,!}$ is the K-theoretic pushforward (Gysin map).
	\end{enumerate}
\end{thm}
\begin{proof}
	As is mentioned above, $H^*(H, \Z)$ is torsion-free thanks to \cite[Corollary 9]{GarciaPradaHeinloth-Duke}.\\
	The Atiyah--Hirzebruch spectral sequence $E_2^{p,q}=H^p(C, \KU^q(\pt))\Rightarrow K^{p+q}(C)$ degenerates at the $E_2$-page for dimension reason. We obtain that
	\begin{itemize}
		\item the Chern character map $(\rk, c_1)\colon \KU^0(C)\xrightarrow{\cong} H^0(C, \Z)\oplus H^2(C, \Z)$ is an isomorphism;
		\item the higher first Chern class\footnote{We use here the topological convention to denote higher Chern classes by half integers. In a more algebro-geometric fashion, it should be denoted by $c_{1,1}$. In general, $c_{j-\frac{1}{2}}$ corresponds to the algebraic notation $c_{j, 1}:\KU^1\to H^{2j-1}$.} map $c_{\frac{1}{2}}\colon \KU^1(C) \xrightarrow{\cong} H^1(C, \Z)$ is an isomorphism. 
	\end{itemize}
	As in the statement of the theorem, let $\{\alpha_i\}_{i=1}^{2g}$ be a basis of $H^1(C, \Z)$. We consider the following basis of $H^0(C, \Z)\oplus H^2(C, \Z)$:
	\[ \alpha_{2g+1}:=1_C+(g-1) [\pt] ; \quad\quad \alpha_{2g+2}:= [\pt].\]
	Via the above Chern class isomorphisms, we lift $\{\alpha_i\}_{i=1}^{2g+2}$ to a basis $\{x_i\}_{i=1}^{2g+2}$ of $\KU^*(C)$.
	In particular, $c_{\frac{1}{2}}(x_i)=\alpha_i$ for $1\leq i\leq 2g$, and $x_{2g+1}$ (resp.~$x_{2g+2}$) is represented by a degree $g-1$ line bundle (resp.~by $[\mathcal{O}_C]-[\mathcal{O}_C(-x)]$ for a point $x\in C$).
	
	We can 
	write the decomposition of the class $[\mathcal{E}]\in \KU^0(C\times H)$ via the K\"unneth formula $\KU^0(C\times H)=\KU^1(C)\otimes \KU^1(H)\oplus \KU^0(C)\otimes \KU^0(H)$ as follows:
	\begin{equation}
	\label{eqn:KunnethK}
	[\mathcal{E}]=\sum_{i=1}^{2g+2}p_1^*(x_i)\otimes p_2^*(e_i),\
	\end{equation}
	where $e_i\in \KU^*(H)$.
	By Markman \cite[Theorem 3]{Markman-IntegralGeneratorPoisson}, the (usual) Chern classes $c_r(e_i)$ for $i=2g+1$ and $2g+2$, and the higher Chern classes $c_{r-\frac{1}{2}}(e_i)$ for $i\leq 2g$, where $r\in \mathbb{N}^*$, generate the cohomology ring $H^*(H, \Z)$. We will show that they all can be written in the form of generators in $(i), (ii), (iii)$ in the statement:
	
	- For the Chern classes of $e_{2g+2}$, it is easy to check that when applying  $p_{2,!}$ to \eqref{eqn:KunnethK}, on the right-hand side, the only contribution is the term $p_{2,!}(p_1^*(x_{2g+2})\otimes p_2^*(e_{2g+2}))=\chi(x_{2g+2})e_{2g+2}=e_{2g+2}$. Indeed,  $\chi(x_{2g+1})=0$ and $\chi(x_{2g+2})=1$ by Riemann--Roch  and the choices for $x_{2g+1}$ and $x_{2g+2}$.  Therefore $e_{2g+2}=p_{2, !}([\mathcal{E}])$, and its Chern classes are included in $(iii)$.
	
	-For the Chern classes of $e_{2g+1}$, it is easy to check that when restricting both sides of \eqref{eqn:KunnethK} to a fiber of $p_1$, on the right-hand side, the only contribution comes from the term $p_1^*(x_{2g+1})\otimes p_2^*(e_{2g+1})$, which restricts to $e_{2g+1}$. Indeed, when restricting to a point of $C$,  $x_{2g+1}$ becomes $1_{\pt}$ and $x_{2g+2}$ becomes zero. Therefore $e_{2g+1}=[\mathcal{E}]|_{H}$, and its Chern classes are included in $(i)$.
	
	-For the higher Chern classes of $\{e_i\}_{i=1}^{2g}$, take the $(1, 2r-1)$-K\"unneth component of the $r$-th Chern classes of both sides of \eqref{eqn:KunnethK}:
	\begin{align*}
	&[c_r(\mathcal{E})]_{(1, 2r-1)}\\
	=&\left[c_r\left(\sum_{i=1}^{2g+2}p_1^*(x_i)\otimes p_2^*(e_i)\right)\right]_{(1, 2r-1)}\\
	=&\left[c_r\left(\sum_{i=1}^{2g}p_1^*(x_i)\otimes p_2^*(e_i)\right)\right]_{(1, 2r-1)}\\
	=& \sum_{i=1}^{2g}\left[c_r\left(p_1^*(x_i)\otimes p_2^*(e_i)\right)\right]_{(1, 2r-1)}\\
	=& \sum_{i=1}^{2g} p_1^*c_{\frac{1}{2}}(x_i)\smile  p_2^*c_{r-\frac{1}{2}}(e_i)\\
	=& \sum_{i=1}^{2g} p_1^*\alpha_i\smile  p_2^*c_{r-\frac{1}{2}}(e_i),
	\end{align*}
	where the second equality follows from the fact that $x_{2g+1}, x_{2g+2}, e_{2g+1}, e_{2g+2}$ are in $\KU^0$ hence do not have higher Chern classes in odd cohomology;
	the third equality is because $c_r\left(\sum_{i=1}^{2g}p_1^*(x_i)\otimes p_2^*(e_i)\right)$ is a polynomial in higher Chern classes of $x_i$'s and $e_i$'s, but each $x_i$ can only contribute to $c_{\frac{1}{2}}(x_i)\in H^1(C, \Z)$ for the $(1, 2r-1)$-K\"unneth component, so only one of $x_i$'s can contribute each time; the fourth equality follows from the explicit computation in \cite{Markman-IntegralGeneratorPoisson}, as is explained in the next lemma; the last equality is by the definition of $x_i$'s. Therefore, $c_{r-\frac{1}{2}}(e_i)$ is the K\"unneth components of $c_r(\mathcal{E})$, which is included in $(ii)$.
	
	To summarize, we have showed the Markman's generators are indeed the ones in $(i), (ii), (iii)$ in the statement.
\end{proof}

\begin{lemma}
	Given a variety $H$ and a curve $C$,  for any $x\in \KU^1(C)$ and $y\in \KU^1(H)$, we have 
	\[\left[c_r(p_1^*x\otimes p_2^*y)\right]_{(1, 2r-1)}=p_1^*c_{\frac{1}{2}}(x)\smile p_2^*c_{r-\frac{1}{2}}(y),\]
	for any positive integer $r$.
\end{lemma}
\begin{proof}
	Combining Markman \cite[Lemma 22 (3)]{Markman-IntegralGeneratorPoisson} and \cite[the last formula in Section 2]{Markman-IntegralGeneratorPoisson}, we get 
	\begin{tiny}
			\[c_r(p_1^*x\otimes p_2^*y)=
		\sum_{\lambda={(1^{m_1}\cdots r^{m_r})}\dashv r}(-1)^{r-\ell(\lambda)}\prod_{i\geq 1}(-1)^{m_i(i-1)}\sum_{0\leq k_1<\cdots<k_{m_i}\leq i-1}\prod_{j=1}^{m_i}{{i-1}\choose{k_j}}p_1^*c_{k_j+\frac{1}{2}}(x)\smile p_2^*c_{i-k_j-\frac{1}{2}}(y).\]
	\end{tiny}
	On the right-hand side, as $c_{k_j+\frac{1}{2}}(x)$ is non-zero only if $k_j=0$. Therefore, \textit{all} $k_j=0$, hence $m_i=0$ or 1 (i.e.~$\lambda$ has no repeated parts), the above formula simplifies to 
	\[c_r(p_1^*x\otimes p_2^*y)=\sum_{\lambda=(\lambda_1>\cdots >\lambda_{\ell})\dashv r}(-1)^{r-\ell}\prod_{j\geq 1}^{\ell}(-1)^{(\lambda_j-1)}p_1^*c_{\frac{1}{2}}(x)\smile p_2^*c_{i-\frac{1}{2}}(y).\]
	Now if we take the $(1, 2r-1)$-K\"unneth component, for each term, there can be at most one $c_{\frac{1}{2}}(x)$ in the product, so only the trivial partition $\lambda=(r)$ contributes to this K\"unneth component, and we get the desired formula. 
\end{proof}

If $C$ is defined over $\R$, then so is $H_C(n, d)$ by construction. In other words, a real structure on $C$ induces a natural real structure on $H_C(n, d)$. We have the following analogue of Theorem \ref{thm:VBAC}.
\begin{thm}\label{thm:Higgs}
	Let $C$ be a smooth projective real curve. Let $n>0$ and $d$ be two coprime integers. If $C$ is maximal, then the moduli space $H_C(n, d)$ is also maximal. Conversely, if $H_C(n, d)$ is maximal and $C(\R)\neq \emptyset$, then $C$ is maximal.
\end{thm}
\begin{proof}
	The proof is completely analogous to the proof of Theorem \ref{thm:VBAC}. We wrote that proof in the way that it applies here verbatim, with $M$ replace by $H$ and we use Theorem \ref{thm:GeneratorHiggs} in the place of Theorem \ref{thm:AtiyahBott}.
\end{proof}


\section{Hilbert squares and Hilbert cubes}
\label{sec:HilbertSquare}
Given  a smooth complex variety $X$ and a positive integer $n$, the Hilbert scheme of length $n$ subschemes on $X$ is called the $n$-th \textit{punctual Hilbert scheme} of $X$, denoted by $X^{[n]}$ or $\Hilb_n(X)$, which is a proper (resp.~projective) scheme if $X$ is proper (resp.~projective). When $\dim(X)=1$, $X^{[n]}$ is nothing but the symmetric power $ X^{(n)}$, which is smooth. When $\dim(X)=2$, $X^{[n]}$ is again smooth by Fogarty \cite{Fogarty} and it provides a crepant resolution of the symmetric power $X^{(n)}$ via the Hilbert--Chow morphism. When $\dim(X)\geq 3$, by Cheah \cite{Cheah}, $X^{[2]}$ and $X^{[3]}$ are still smooth, but $X^{[n]}$ are singular for $n\geq 4$.

More generally, for any sequence of positive integers $n_1<n_2<\cdots < n_r$, we have the so-called \textit{nested} punctual Hilbert scheme $X^{[n_1, n_2, \cdots, n_r]}$ classifying flags of subschemes of lengths $n_1, \dots, n_r$, defined as the incidence subscheme in the product $X^{[n_1]}\times \cdots\times X^{[n_r]}$. Cheah \cite{Cheah} proved that when $\dim(X)\geq 3$, the only smooth nested Hilbert schemes are $X$, $X^{[1, 2]}$, $X^{[2]}$, $X^{[3]}$ and $X^{[2,3]}$.

\begin{rmk}[Explicit construction]
	\label{rmk:ExplicitConstruction}
	We call $X^{[2]}$ the \textit{Hilbert square} of $X$. One can construct $X^{[2]}$ simply as the blow-up of the (in general singular) symmetric square $X^{(2)}$ along the diagonal.  An alternative way of construction is as follows: let $\Bl_{\Delta}(X\times X)$ be the blow-up of $X\times X$ along its diagonal, then the natural involution on $X\times X$ lifts to the blow-up and the quotient of $\Bl_{\Delta}(X\times X)$ by the involution is isomorphic to $X^{[2]}$.\\
	Similarly, we call $X^{[3]}$ the \textit{Hilbert cube} of $X$. There is also a concrete construction for the Hilbert cube starting from $X\times X\times X$. Roughly, we first perform successive blow-ups in $X^3$ to resolve the natural rational map $X^3\dashrightarrow X^{[3]}$, then quotient by the symmetric group $\mathfrak{S}_3$, and finally contract a divisor; see \cite{ShenVial-Hilb3} for the details. 
\end{rmk}
\begin{rmk}[Douady space]
The construction of Hilbert schemes can be generalized to the analytic category: for a complex manifold $X$ and a positive integer $n$, we have the so-called $n$-th punctual Douady space $X^{[n]}$, as well as the nested version. The aforementioned results of Cheah still hold in this setting: the only smooth nested punctual Douady spaces associated to a complex manifold $X$ of dimension $\geq 3$ are $X$, $X^{[1, 2]}$, $X^{[2]}$, $X^{[3]}$ and $X^{[2,3]}$. In what follows, we will use the algebraic terminology Hilbert schemes (squares, cubes, etc.) to refer to their Douady analogues, and all the results in this section hold in the analytic category as well as in the algebraic category, with the same proof.
\end{rmk}

If $(X, \sigma)$ is a real variety, then $\sigma$ naturally induces real structures on the nested punctual Hilbert schemes. Indeed, for a finite length closed subscheme $Z$ of $X$, 
consider the base change by the conjugate automorphism of the base field $\C$:
\begin{equation*}
\xymatrix{
	&\bar Z \ar@{^{(}->}[dl]_{\sigma\circ i'}\cart\ar[r]^{\conjug}\ar@{^{(}->}[d]^{i'}& Z\ar@{^{(}->}[d]^{i}\\
	X\ar[dr]_{f}&\bar X\ar[d]^{f'}\cart\ar[l]^{\sigma} \ar[d] \ar[r]^{\conjug}& X\ar[d]^{f}\\
	&\Spec(\C) \ar[r]^{\conjug} &\Spec(\C)
}
\end{equation*}
then define the image of $Z$ to be the closed subscheme $\sigma\circ i': \bar{Z}\hookrightarrow S$. 
In the sequel, we always equip the (nested) Hilbert schemes with such natural real structures. 

\begin{thm}
	\label{thm:HilbertSquareCube}
	Let $X$ be a smooth real variety. 
	\begin{enumerate}
		\item[(i)] If $X$ is maximal, then $X^{[1, 2]}$ is maximal.
		\item[(ii)] If $X^{[2]}$ is maximal and $X(\R)\neq \emptyset$, then $X$, $X^{[1, 2]}$, $X^{[2,3]}$  and $X^{[3]}$ are all maximal.
	\end{enumerate} 
\end{thm}
\begin{proof}
(i) It is easy to see that $X^{[1,2]}$ is isomorphic to $\Bl_{\Delta}(X\times X)$, the blow-up of $X\times X$ along its diagonal, whose maximality follows from Lemma \ref{lemma:Product} and Proposition \ref{prop:Blowup}.\\
(ii) Denote by $\tau\colon \Bl_{\Delta}(X\times X)\to X\times X$ and $g\colon \Bl_{\Delta}(X\times X)\to X^{[2]}$ the natural morphisms, both defined over $\R$. For any class $\gamma\neq [X]\in H^*(X, \F2)$, we first remark that 
\[\tau_*g^*g_*\tau^*(\gamma\otimes [X])=\gamma\otimes [X]+[X]\otimes \gamma \in H^*(X\times X, \F2).\]
Since $g_*\tau^*(\gamma\otimes [X])\in H^*(X^{[2]}, \F2)$, which is the image of  some $\alpha\in H^*_G(X^{[2]}, \F2)$ by the maximality assumption, we obtain that
$$\gamma\otimes [X]+[X]\otimes \gamma=\tau_*g^*(\alpha)\in \im \left(H^*_G(X\times X, \F2)\to H^*(X\times X, \F2)\right).$$
Take any real point $p\in X(\R)$, define the $\R$-morphism $i_p\colon X\to X\times X$ sending $x$ to $(x, p)$. It follows that $i_p^*(\gamma\otimes [X]+[X]\otimes \gamma)=\gamma$ is in the image of $H^*_G(X, \F2)\to H^*(X, \F2)$. Therefore $X$ is maximal by Proposition \ref{prop:MaximalCriterion}. (The maximality of $X$ in the surface case is also shown by Kharlamov--R{\u{a}}sdeaconu in \cite[Theorem 1.1]{Kharlamov-Rasdeaconu-HilbertSquare} with a different argument and without the assumption $X(\mathbb{R})\neq\emptyset$.)

Using (i), $X^{[1,2]}$ is maximal as well. 

 For $X^{[2,3]}$, we first note that the universal subscheme in $X\times X^{[2]}$ is canonically identified with the nested Hilbert scheme $X^{[1,2]}$, which is in turn isomorphic to $\Bl_{\Delta}(X\times X)$ as mentioned above. 
It is known that $X^{[2,3]}$ is isomorphic to the blow-up of $X\times X^{[2]}$ along its universal subscheme. Hence 
\[X^{[2,3]}\cong \Bl_{X^{[1,2]}}(X\times X^{[2]}).\]
All the involved embeddings are real morphisms between smooth real varieties. Therefore we can conclude the maximality of $X^{[2,3]}$ by invoking Lemma \ref{lemma:Product}, Proposition \ref{prop:Blowup} together with the maximalities of $X$, $X^{[2]}$ and $X^{[1,2]}$.

Finally, for $X^{[3]}$, note that there is a natural surjection $\pi\colon X^{[2,3]}\to X^{[3]}$ by forgetting the length 2 subscheme. It is easy to see that $\pi$ is generically finite of degree $3$ (an odd number!). Proposition \ref{prop:SurjectionOddDegree} allows us to deduce the maximality of $X^{[3]}$ from that of $X^{[2,3]}$.
\end{proof}

\begin{rmk}[Loss of maximality in Hilbert squares]
	\label{rmk:LossOfMaximality}
	It is quite surprising that Kharlamov and R{\u{a}}sdeaconu \cite{Kharlamov-Rasdeaconu-HilbertSquare} recently discovered the existence of maximal real surfaces with non-maximal Hilbert squares. In fact, they showed in \cite[Theorem 1.2]{Kharlamov-Rasdeaconu-HilbertSquare} that for a real surface $X$ with $H^1(X, \F2)=0$, its Hilbert square $X^{[2]}$ is maximal if and only if $X(\R)$ is connected. This finding led to the discovery of many examples of maximal surfaces with non-maximal Hilbert square, such as K3 surfaces. We will provide in Theorem \ref{thm:NonMaxHilbCubic4} a higher-dimensional counter-example: the Hilbert square of a maximal real cubic fourfold is \textit{never} maximal.
\end{rmk}

Although taking Hilbert square fails to preserve maximality in general, we provide an interesting example where Theorem \ref{thm:HilbertSquareCube} can be applied. More examples in the surface case can be found in Section \ref{sec:HilbertSchemes}. 
\begin{thm}[Cubic threefolds]
	\label{thm:Cubic3folds}
	Let $X$ be a smooth real cubic threefold. If $X$ is maximal, then $X^{[2]}$, $X^{[3]}$, $X^{[1,2]}$, $X^{[2,3]}$ are all maximal.
\end{thm}
\begin{proof}
	By Krasnov \cite{Krasnov-Cubic}, the maximality of $X$ implies that its Fano surface of lines $F:=F(X)$ is also maximal (see Proposition \ref{prop:Cubic3}). Applying the realization functors $H^*(-, \F2)$ and $H^*(-(\mathbb{R}), \F2)$ to the motivic Galkin--Shinder relation \cite{GalkinShinder} established in \cite[Corollary 18]{BelmansFuRaedschelders}, we obtain the following relations between the total $\F2$-Betti numbers of $X$, $F$ and $X^{[2]}$:
	\begin{align*}
		b_*(F, \F2)+4 b_*(X, \F2)&=b_*(X^{[2]}, \F2),\\
		b_*(F(\mathbb{R}), \F2)+4 b_*(X(\mathbb{R}), \F2)&=b_*(X^{[2]}(\mathbb{R}), \F2).
	\end{align*}
	Therefore, $X^{[2]}$ is also maximal. We can conclude by applying Theorem \ref{thm:HilbertSquareCube}.
\end{proof}

\section{Punctual Hilbert schemes of surfaces}
\label{sec:HilbertSchemes}
As discussed in the previous section, for a smooth complex surface $S$, the punctual Hilbert scheme/Douady space $S^{[n]}$ is smooth for any $n\geq 1$ (\cite{Fogarty}). Moreover, a real structure on $S$ naturally induces a real structure on $S^{[n]}$. This section aims to prove the maximality of Hilbert powers of certain maximal $\mathbb{R}$-surfaces:

\begin{thm}\label{thm:HilbertPower}
	Let $(S, \sigma)$ be a maximal smooth projective real surface  with $H^1(S, \F2)=0$.
	If $S(\mathbb{R})$ is connected, then $S^{[n]}$ is  maximal for any $n\geq 1$.
\end{thm}

\begin{rmk} 
		\label{rmk:ConditionsOnS}
		The case of $n=2$ in Theorem \ref{thm:HilbertPower} gives a different proof of the ``if'' part of Kharlamov--R{\u a}sdeaconu's result \cite[Theorem 1.2]{Kharlamov-Rasdeaconu-HilbertSquare}.  It is worth noting some details about the conditions in Theorem \ref{thm:HilbertPower}.
	\begin{itemize}
		\item The condition on the connectedness of $S(\mathbb{R})$ can be equivalently replaced by $\rk H^2(S, \mathbb{Z})^{\sigma}=0$, or by $c_1$-maximality (Definition \ref{def:c1Max}), as shown in Proposition \ref{prop:MaximalAndC1Maximal}.
		\item By Poincar\'e duality and the universal coefficient theorem, the condition $H^1(S, \F2)=0$ is equivalent to $b_{\operatorname{odd}}(S)=0$ and  $H^*(S, \Z)$ has no 2-torsion. By G\"ottsche \cite{Goettsche} and Totaro \cite[Theorem 3.1]{TotaroHilbn}, this condition implies that for any $n\geq 1$, the odd Betti numbers of $S^{[n]}$ vanish and the integral cohomology ring of $S^{[n]}$ is 2-torsion-free.
		\item All the conditions in Theorem \ref{thm:HilbertPower} are invariant under blowups at real points.
	\end{itemize}
\end{rmk}

\begin{rmk}[Nested Hilbert schemes]
	By Theorem \ref{thm:HilbertSquareCube}, for any maximal real surface $S$ satisfying the conditions in Theorem \ref{thm:HilbertPower}, the nested Hilbert schemes $S^{[1,2]}$ and $S^{[2,3]}$ are also maximal. However, by Cheah \cite{Cheah}, in dimension 2, all the nested Hilbert schemes $S^{[n,n+1]}$ are smooth.  This raises the interesting question whether $S^{[n, n+1]}$ is maximal for any $n$ in this case.
\end{rmk}

Our strategy for proving Theorem \ref{thm:HilbertPower} is similar to that of Theorem \ref{thm:VBAC}. We will rely on the following result of W.~Li and Z.~Qin \cite[Theorem 1.2]{LiQin08} (based on \cite{LiQinWang-MathAnn02, LiQinWang-IMRN02, LiQinWang-Crelle03, QinWang05}) on the generators of the cohomology ring of $S^{[n]}$, which can be viewed as the analogue of Atiyah--Bott's Theorem \ref{thm:AtiyahBott}.

\begin{thm}[Li--Qin]
	\label{thm:LiQin}
Let $S$ be a smooth projective complex surface with $b_1(S)=0$, and $n$ a positive integer. Then $H^*(S^{[n]}, \Z)/\Tors$ is generated by the following three types of elements:
\begin{enumerate}
\item[(i)] For $1\leq r\leq n$, the Chern class $c_r(\mathcal{O}_S^{[n]})\in H^{2r}(S^{[n]}, \Z)$, where $\mathcal{O}_S^{[n]}$ is the tautological rank-$n$ bundle defined as $p_*(\mathcal{O}_{\mathcal{Z}_n})$ where $p\colon S^{[n]}\times S\to S^{[n]}$ is the projection and $\mathcal{Z}_n\subset S^{[n]}\times S$ is the universal subscheme.
\item[(ii)] For $1\leq r\leq n$, and $\alpha\in H^2(S, \Z)/\Tors$,  the class $\mathds{1}_{-(n-r)}\mathfrak{m}_{(1^r), \alpha}|0\rangle$.
\item[(iii)] For $1\leq r\leq n$,  the class $\mathds{1}_{-(n-r)}\mathfrak{p}_{-r}(\pt)|0\rangle$, where $\pt\in H^4(S, \Z)$ denotes the class of a point. 
\end{enumerate}
Here $|0\rangle$ denotes the positive generator of $H^0(S^{[0]}, \Z)$, the operator $\mathfrak{p}$  is Nakajima's operator \cite{Nakajima}, and the operators $\mathds{1}$, $\mathfrak{m}$ are defined in \cite{LiQinWang-Crelle03, QinWang05, LiQin08}; see the explanation in the remark below.
\end{thm}

\begin{rmk}[Operators]
	\label{rmk:Operators}
	Let us give some more details on the operators $\mathfrak{p}$, $\mathds{1}$ and $\mathfrak{m}$. In the discussion below, the coefficients of the cohomology can be chosen arbitrarily (for example $\Z$), and are omitted.
	\begin{enumerate}
		\item For any $r\geq 0$ and any $\alpha\in H^*(S)$, the (creation) Nakajima operator $\mathfrak{p}_{-r}(\alpha)$ sends, for any $j\geq 0$, an element $\beta\in H^*(S^{[j]})$ to the class $q_*(p^*(\beta)\smile\rho^*(\alpha)\smile [Q_{j+r, j}])\in H^*(S^{[j+r]})$, where $Q_{j+r, j}:=\{(Z, x, Z')\in S^{[j]}\times S\times S^{[j+r]}~|~ \operatorname{supp}(I_Z/I_{Z'})=\{x\}\}$, and $p, \rho, q$ are the projections from $S^{[j]}\times S\times S^{[j+r]}$ to $S^{[j]}$, $S$ and $S^{[j+r]}$ respectively.
		In particular, taking $j=0$, then $\mathfrak{p}_{-r}(\alpha)|0\rangle$ is the image of $\alpha$ via the correspondence $[\Gamma_r]_*\colon H^*(S)\to H^*(S^{[r]})$, where		
		$\Gamma_r:=\{(x, Z)\in S\times S^{[r]}~|~ \operatorname{supp}(Z)=\{x\}\}$.
		\item For any $r\geq 0$, the operator $\mathds{1}_{-r}$ is defined in \cite[Definition 4.1]{LiQinWang-Crelle03}, \cite[(2.7)]{LiQin08} as $\mathds{1}_{-r}:=\frac{1}{r!}\mathfrak{p}_{-1}(1_S)^r$. Despite of the apparent denominator, in \cite[Lemma 3.3]{QinWang05} it is proved that $\mathds{1}_{-r}$ is an \textit{integral} operator. Indeed, it can be equivalently defined as follows: for any $j\geq 0$, the operator $\mathds{1}_{-r}: H^*(S^{[j]})\to H^*(S^{[j+r]})$ is the correspondence $[S^{[j, j+r]}]_*$, where $S^{[j, j+r]}:=\{ (Z, Z') \in S^{[j]}\times S^{[j+r]}~|~Z\subset Z'\}$ is the nested Hilbert scheme. 
		\item The operator $\mathfrak{m}$ is defined in \cite[\S 4.2]{QinWang05} and in \cite[Definition 3.3]{LiQin08}. The precise definition is somewhat involved. Let us only mention here that when $\alpha=[C]$ for a smooth irreducible curve $C$ in $S$, then for any partition $\lambda=(\lambda_1\geq \cdots \geq \lambda_N)$ of $n$, we have $\mathfrak{m}_{\lambda, C}|0\rangle= [L^\lambda C]$, where $L^\lambda C$ is the closure in $S^{[n]}$ of $\{\lambda_1x_1+\cdots+\lambda_Nx_N | x_i \in S \text{ distinct}\}.$ An interpretation which works for $\lambda=(1^n)$ and for any class $\alpha$, will be given and used in the proof of Theorem \ref{thm:HilbertPower} below.
		
	\end{enumerate}
\end{rmk}

\begin{proof}[Proof of Theorem \ref{thm:HilbertPower}]
	By Remark \ref{rmk:ConditionsOnS}, the condition $H^1(S, \F2)=0$ implies that $b_{\operatorname{odd}}(S^{[n]})=0$ and $H^*(S^{[n]}, \Z)$ has no 2-torsion. By the universal coefficient theorem, the cohomology ring $H^*(S^{[n]}, \F2)$ is generated by the same collection of generators in Theorem \ref{thm:LiQin} by passing to $\F2$-coefficients. Like in Theorem \ref{thm:VBAC}, we only need to show that all the generators are in the image of the restriction map $H^*_G(S^{[n]}, \F2)\to H^*(S^{[n]}, \F2)$.

	- For generators in $(i)$, since the tautological vector bundle $\mathcal{O}_S^{[n]}$ is a holomorphic real vector bundle on $S^{[n]}$ by construction, hence defines a class $[\mathcal{O}_S^{[n]}]\in \KR^0(S^{[n]})$. As $c^G_r(\mathcal{O}_S^{[n]})\in H^*_G(S^{[n]}, \F2)$ is mapped to $c_r(\mathcal{O}_S^{[n]}) \in H^*(S^{[n]}, \F2)$, the latter is in the image of the restriction map.
	
	- For generators in $(ii)$, given a class $\alpha\in H^2(S, \F2)$,
	by Proposition \ref{prop:MaximalAndC1Maximal}, $S$ is $c_1$-maximal, i.e.~$c_1\colon \KR^0(S)\to H^2(S, \F2)$ is surjective, there exists an element $L\in \KR^0(S)$ such that $c_1(L)=\alpha$. Up to replacing $L$ by its determinant, one can in fact  assume that $L$ is a real vector bundle of rank 1.
	
	We denote  $\widetilde{\alpha}:=c_1^G(L)\in H^2_G(S, \F2)$ the equivariant first Chern class of $L$. Clearly, $\widetilde{\alpha}$ is mapped to $\alpha$ under the restriction map.
	Let  $\mathcal{Z}_n\subset S^{[r]}\times S$ be the universal subscheme. 
	Let $p_1, p_2$ be the two natural projections from $\mathcal{Z}_r$ to $S^{[r]}$ and $S$, which are real morphisms since $\mathcal{Z}_n$ is a subscheme defined over $\R$.
	We thus obtain an element $p_{1,!}(p_2^*(L))\in \KR^0(S^{[n]})$. By the proof of \cite[Lemma 3.5]{LiQin08},
	 we have the following equality (which holds integrally):
	  $$c_r(p_{1,!}(p_2^*(L)))=\mathfrak{m}_{(1^r), \alpha}\in H^*(S^{[r]}, \F2).$$ Therefore, $c^G_r(p_{1,!}(p_2^*(L)))\in H_G^*(S^{[r]}, \F2)$ is mapped to $\mathfrak{m}_{(1^r), \alpha}$ via the restriction map. 
	
	Now consider the following commutative diagram
	\begin{equation}
	\xymatrix{
		 H^*_G(S^{[r]}, \F2)\ar[r]^{[S^{[r,n]}]_*}\ar[d]& H^*_G(S^{[n]}, \F2)\ar[d]\\
		H^*(S^{[r]}, \F2)\ar[r]^{[S^{[r,n]}]_*}& H^*(S^{[n]}, \F2)
	}
	\end{equation}
	where the vertical arrows are restriction maps, and the horizontal maps are given by the correspondences induced by the nested Hilbert scheme $S^{[r,n]}:=\{(Z, Z')\in S^{[r]}\times S^{[n]}~|~ Z \subset Z'\}$, viewed as algebraic cycles (defined over $\R)$. By Remark \ref{rmk:Operators}, the bottom arrow sends $\mathfrak{m}_{(1^r), \alpha}$ to $\mathds{1}_{-(n-r)}\mathfrak{m}_{(1^r), \alpha}$. Since we just proved that $\mathfrak{m}_{(1^r), \alpha}$ is in the image of the left vertical arrow, we conclude by the commutativity of the diagram that $\mathds{1}_{-(n-r)}\mathfrak{m}_{(1^r), \alpha}$ is in the image of the right vertical arrow.
		
	- For generators in $(iii)$, by definition, $\mathds{1}_{-(n-r)}\mathfrak{p}_{-r}(\pt)|0\rangle$ is the image of the class $[\pt]\in H^4(S)$ under the following composition of correspondences:
	\[H^*(S)\xrightarrow{[\Gamma_r]_*} H^*(S^{[r]})\xrightarrow{[S^{[r,n]}]_*} H^*(S^{[n]}),\]
	where $\Gamma_r:=\{(x, Z)\in S\times S^{[r]}~|~ \operatorname{supp}(Z)=\{x\}\}$, $S^{[r,n]}$ is as above, both viewed as algebraic cycles (see Remark \ref{rmk:Operators}). However, $\Gamma_n$ and $S^{[r,n]}$ (although singular) are defined over $\R$, they induce equivariant correspondences; we have the following commutative diagram:
	\begin{equation}
		\xymatrix{
			H^*_G(S, \F2)\ar[r]^{[\Gamma_r]_*} \ar@{->>}[d]& H^*_G(S^{[r]}, \F2)\ar[r]^{[S^{[r,n]}]_*}\ar[d]& H^*_G(S^{[n]}, \F2)\ar[d]\\
			H^*(S, \F2)\ar[r]^{[\Gamma_r]_*} & H^*(S^{[r]}, \F2)\ar[r]^{[S^{[r,n]}]_*}& H^*(S^{[n]}, \F2)
		}
	\end{equation}
	where all the vertical arrows are restriction maps and the left one is surjective since $S$ is maximal (Lemma \ref{lemma:RelationMaximalities}). Therefore, $\mathds{1}_{-(n-r)}\mathfrak{p}_{-r}(\pt)|0\rangle=[S^{[r,n]}]_*\circ [\Gamma_r]_*([\pt])$ is in the image of the restriction map. 
	
	To summarize, all types of generators lie in the image of the restriction map $H^*_G(S^{[n]}, \F2)\to H^*(S^{[n]}, \F2)$, hence the latter is surjective, and $S^{[n]}$ is maximal by Proposition \ref{prop:MaximalCriterion}.
\end{proof}

We provide two types of surfaces to which Theorem \ref{thm:HilbertPower} applies: rational surfaces and some surfaces of Kodaira dimension 1.
\subsection{Example I: rational surfaces}

The first type of examples includes $\mathbb{P}^2$, $\mathbb{P}^1\times \mathbb{P}^1$, and Hirzebruch surfaces, as well as blow-ups of these surfaces at real points.
\begin{cor}
	\label{cor:RealRationalSurface}
The punctual Hilbert schemes of a maximal $\R$-rational real surface are maximal.
\end{cor}
\begin{proof}
	By Theorem \ref{thm:HilbertPower}, we only need to verify that $H^1(S, \F2)=0$ and that $S(\mathbb{R})$ is connected (or equivalently, $c_1$-maximal). The former is obvious for rational surfaces, and the latter follows from Lemma \ref{lemma:KMaxRationalSurface} and Lemma \ref{lemma:RelationMaximalities}.
\end{proof}

\begin{rmk}[Del Pezzo of degree 1]
	If we consider a more general class of surfaces, namely \textit{geometrically rational} surfaces, which are real surface $S$ with $S_{\C}$ rational, there is another type of maximal minimal real surface, namely $\mathbb{B}_1$, the real minimal del Pezzo surface of degree 1 whose real locus is $\mathbb{R}\mathbb{P}^2\coprod4 \mathbb{S}^2$, see \cite{Kollar-TopologyOfRealVar} \cite{Russo-RealDelPezzo}. By Proposition \ref{prop:MaximalAndC1Maximal}, $\mathbb{B}_1$ is not $c_1$-maximal (nor K-maximal).  According to the result of Kharlamov--R{\u{a}}sdeaconu \cite{Kharlamov-Rasdeaconu-HilbertSquare} mentioned in Remark \ref{rmk:LossOfMaximality}, the Hilbert square of $\mathbb{B}_1$ is not maximal. Furthermore, note that $\mathbb{B}_1$ is not Hodge expressive in the sense of \cite{BrugalleSchaffhauser}.
\end{rmk}

\subsection{Example II: generalized Dolgachev surfaces}
Let $m>1$ be an integer, and choose another auxiliary integer $0< k<m$.  Recall that for a relatively minimal elliptic surface $f\colon S\to C$ with a smooth fiber $S_c$ over a point $c\in C$, a \textit{logarithmic transformation} of order $m$ at $c$ produces another relatively minimal elliptic surface $f'\colon S'\to C$ with a multiple fiber $S'_c$ of multiplicity $m$ over $c$, such that $S\backslash S_c$ is isomorphic to $S'\backslash S'_c$ over $C\backslash\{c\}$. In fact, the logarithmic transformation can be performed more generally to replace a fiber of type $I_b$ in an elliptic fibration by a multiple fiber of type ${}_mI_b$. We refer to \cite[Chapter V, \S 13]{BHPV} for more details.

Dolgachev \cite{DolgachevSurface} used logarithmic transformations to construct simply connected irrational elliptic surfaces of Kodaira dimension 1 with vanishing geometric genus. Although all homeomorphic to the blow-up of $\mathbb{P}^2$ at 9 points \cite{Freedman-JDG82}, Donaldson \cite{Donaldson-JDG87} and Okonek--Van de Ven \cite{OkonekVandeVen} showed that Dolgachev surfaces provide infinitely many distinct diffeomorphic types.

For our purpose, we follow L\"{u}bke--Okonek \cite{Lubke-Okonek} to construct certain variants of Dolgachev surfaces and adapt them to real geometry. We consider two real smooth planar cubic curves intersecting at 9 distinct real points in $\mathbb{P}^2(\mathbb{R})$. Blowing up these 9 points, we get an $\mathbb{R}$-rational surface $S$ fibered in elliptic curves over $ \mathbb{P}^1$. The morphism 
\[f\colon S\to \mathbb{P}^1\]
is defined over $\mathbb{R}$.
For any odd integer $m>2$ and a generic pair of complex conjugate points $t, \bar{t} \in \mathbb{P}^1(\mathbb{C})\backslash \mathbb{P}^1(\mathbb{R})$, performing the logarithmic transformations at $t$ and $\bar{t}$ of the same order $m$ (and with the same auxiliary choice of $k$) produces another elliptic surface $S_{m,m}\to \mathbb{P}^1$ which is isomorphic to $S$ over $\mathbb{P}^1\backslash\{t, \bar{t}\}$.  The real structure on $S$ gives rise to a real structure on $S_{m,m}$. 

\begin{lemma}
	\label{lemma:DolgachevSurface}
The real locus of $S_{m,m}$ is $$S_{m,m}(\mathbb{R})\cong \underbrace{\mathbb{R}\mathbb{P}^2\#\cdots\#\mathbb{R}\mathbb{P}^2}_{10 \text{ copies}},$$
with $\F2$-Betti numbers 1, 10, 1. 

The $\F2$-Betti numbers of $S_{m,m}$ are 1, 0, 10, 0, 1. 
In particular, $S_{m, m}$ is a maximal real surface.
\end{lemma}
\begin{proof}
	Since the logarithmic transformation along two conjugate complex points does not change the real locus, $S_{m,m}(\mathbb{R})$ is the same as the real locus of the blow-up of $\mathbb{P}^2$ at 9 real points. As blowing up at a real point modifies the real locus by taking a connected sum with a copy of $\mathbb{R}\mathbb{P}^2$, the formula of $S_{m,m}(\mathbb{R})$ follows. Its $\F2$-Betti numbers can be easily computed.
	
	Contrary to the simple connectedness of the classical Dolgachev surfaces, we know that $\pi_1(S_{m,m})\cong\mathbb{Z}/m\mathbb{Z}$ (see \cite{DolgachevSurface} or \cite{Lubke-Okonek}). Since $m$ is odd, $H^1(S_{m,m}, \F2)=0$. The Euler characteristic is not changed by logarithmic transformations (hence is equal to $\chi_{\top}(\Bl_{9}\mathbb{P}^2)=12$), and therefore the $\F2$-Betti numbers of $S_{m,m}$ can be determined.
\end{proof}

\begin{cor}[Generalized Dolgachev surface]
	\label{cor:Dolgachev}
	The punctual Hilbert schemes of the real surfaces $S_{m,m}$ constructed above are maximal.
\end{cor}
\begin{proof}
By Lemma \ref{lemma:DolgachevSurface}, the surface $S_{m,m}$ satisfies all the conditions in Theorem \ref{thm:HilbertPower}, which implies that its punctual Hilbert schemes are maximal. 
\end{proof}

\begin{rmk}
	The surfaces $S_{m,m}$ provide infinitely many new examples as they have Kodaira dimension 1 (see \cite[P. 219]{Lubke-Okonek}) and
	by \cite[P. 220, Corollary]{Lubke-Okonek}, $\{S_{m,m}\}_m$ contains infinitely many diffeomorphic types.
\end{rmk}

\section{Moduli spaces of sheaves on the projective plane}
\label{sec:ModuliP2}
Given integers $r>0, c_1, c_2$, let $M:=M(r, c_1, c_2)$ be the moduli space of stable sheaves of rank $r$ with first and second Chern classes $c_1$ and $c_2$ on the complex projective plane $\PP^2_{\C}$. Assume that 
\begin{equation}\label{eqn:NumericalCondP2}
	\gcd\left(r, c_1, \frac{c_1(c_1+1)}{2}-c_2\right)=1,
\end{equation}
then $M$ is a smooth projective fine moduli space.

We equip $\PP^2_{\C}$ with its standard real structure $\sigma$. We claim that this induces a natural real structure  $\sigma_M$ on $M$:
\begin{align*}
	\sigma_M\colon M &\to M\\
	E & \mapsto \overline{\sigma^*(E)}.
\end{align*}
Indeed, $\sigma_M$ is clearly anti-holomorphic of order 2, we only need to show that the Chern classes of $\overline{\sigma^*(E)}$ and $E$ are the same: first, it is easy to see that $\sigma^*$ acts on cohomology as $\id$ on $H^0(\PP^2_{\C}, \Z)\oplus H^4(\PP^2_{\C}, \Z)$, and as $-\id$ on $H^2(\PP^2_{\C}, \Z)$ (because complex conjugate changes of orientation on $\C$). Therefore, 
\[c_i(\overline{\sigma^*(E)})=(-1)^i\sigma^*(c_i(E))=(-1)^i(-1)^ic_i(E)=c_i(E).\]
The claim is proved and $\sigma_M$ is a real structure. 

Recall that $\PP^2_{\C}$ has only one real structure (up to equivalence), namely the standard one, and it is maximal ($b_*(\PP^2(\C), \F2)=b_*(\PP^2(\R), \F2)=3$). We show that the maximality is inherited by the moduli spaces:
\begin{thm}
	\label{thm:ModuliP2}
	For integers $r>0, c_1, c_2$ satisfying \eqref{eqn:NumericalCondP2}, the moduli space $M(r, c_1, c_2)$ of stable sheaves on the projective plane equipped with its natural real structure is maximal.
\end{thm}
\begin{proof}
	By construction (using GIT for example), $M:=M(r, c_1, c_2)$ is a smooth projective variety defined over $\R$.
	Since $M$ is a fine moduli space, there is a universal sheaf $\mathcal{E}$ on $M\times \PP^2$, which is also defined over $\R$. 
	By Ellingsrud and Stromme \cite{EllingsrudStromme}, the cohomology ring $H^*(M, \Z)$ has no odd cohomology, is torsion-free and generated by the elements
	\begin{equation}
		\label{eqn:GeneratorsForP2}
		c_i(R^1p_*(\mathcal{E}(-j))),
	\end{equation}
	where $p:M\times S\to M$ is the natural projection, $1\leq j\leq 3$, $i\in \mathbb{N}$.
	Therefore $H^*(M, \F2)$ is also generated, as $\F2$-algebra, by (the images of) the elements in \eqref{eqn:GeneratorsForP2}. However, $R^1p_*(\mathcal{E}(-j))$ is a real algebraic vector bundle on $M$, hence its equivariant Chern class $c^G_i(R^1p_*(\mathcal{E}(-j)))\in H^*_G(M, \F2)$ is mapped to $c_i(R^1p_*(\mathcal{E}(-j)))$ via the restriction map
	$H^*_G(M, \F2)\to H^*(M, \F2)$. Therefore the restriction map is surjective and $M$ is maximal by Proposition \ref{prop:MaximalCriterion}.
\end{proof}

\section{Moduli spaces of sheaves on Poisson surfaces}
\label{sec:Poisson}
Let $S$ be a Poisson surface, that is, a smooth projective complex surface with an effective anti-canonical bundle: $H^0(S, -K_S)\neq 0$. Poisson surfaces are classified by Bartocci and Macr\`i \cite{BartocciMacri-PoissonSurface}: they are either symplectic (i.e.~K3 or abelian) surfaces, or ruled surfaces; see \cite[Theorem 1.1]{BartocciMacri-PoissonSurface} for the precise classification of ruled Poisson surfaces.

A Mukai vector on $S$ is an element $v\in \KU^0(S)$. We denote by $\rk(v)\in \Z$ its rank, and by $c_i(v)\in H^{2i}(S, \Z)$ its $i$-th Chern class for $i=1, 2$. Any coherent sheaf $E$ on $S$ gives rise to a class in the Grothendieck group $[E]\in \operatorname{K}^0(S)$, its image under the natural map $\operatorname{K}^0(S)\to \KU^0(S)$ is called the Mukai vector of $E$. A polarization $H$ is called \textit{$v$-generic}, if $H$-stability coincides with $H$-semi-stability for any coherent sheaves with Mukai vector $v$. Here and in the sequel,  (semi-)stability is in the sense of Gieseker and Maruyama (see \cite{HuybrechtsLehn}). 

Let $v$ be a primitive Mukai vector with $\rk(v)> 0$ and $c_1(v)\in \operatorname{NS}(S):=H^2(S, \Z)\cap H^{1,1}(S)$, and  $H$ a $v$-generic polarization, which always exists\footnote{Here we assumed $\rk(v)>0$ for simplicity. The case $\rk(v)=0$ can also be treated with some extra care. For example, in Markman \cite[P. 628, Condition 7]{Markman-IntegralGeneratorPoisson}, a further assumption is added when $S$ is a ruled Poisson surface and $\rk(v)=0$.}.  Let $M:=M_H(S, v)$ be the moduli space of $H$-stable sheaves on $S$ with Mukai vector $v$.

For a symplectic surface $S$, the moduli space $M$ is a connected smooth projective holomorphic symplectic variety of dimension $2n:=2-\chi(v, v)=(1-r)c_1^2+2rc_2-r^2 \chi(\mathcal{O}_S)+ 2$, and is deformation equivalent to the $n$-th punctual Hilbert scheme $S^{[n]}$ (hence hyper-K\"ahler if $S$ is a K3 surface). This is a theorem established by the works of Mukai \cite{Mukai84}, G\"ottsche--Huybrechts \cite{GottscheHuybrechts}, O'Grady \cite{OGrady97}, Yoshioka \cite{Yoshioka-Crelle99} etc; see \cite[Theorem 6.2.5]{HuybrechtsLehn}.

For a ruled Poisson surface $S$, the moduli space $M$, if non-empty, is a connected smooth projective variety of dimension $1-\chi(v, v)=(1-r)c_1^2+2rc_2-r^2 \chi(\mathcal{O}_S)+ 1$  with a Poisson structure. This is due to Bottacin \cite{Bottacin-Poisson} (based on Tyurin \cite{Tyurin-SymplecticStructure}).

\begin{thm}
	\label{thm:PoissonSurface}
	Let $(S, \sigma)$ be a Poisson surface with a real structure. Let $v\in \KU^0(S)$ be a primitive  Mukai vector with $\rk(v)>0$ and satisfying $c_1(v)\in \operatorname{NS}(S)$ and  $\sigma^*(c_1(v))=-c_1(v)$. Let $H$ be a real algebraic ample line bundle that is $v$-generic. Then $\sigma$ induces a natural real structure on $M:=M_H(S, v)$. If $S$ is K-maximal (Definition \ref{def:KMax}) in the sense that the natural maps $\KR^i(S)\to \KU^i(S)$ are surjective for $i=0, 1$, then $M$ is maximal.
\end{thm}
\begin{proof}
		Let $\mathcal{E}$ be a (quasi)-universal sheaf on $M\times S$, which defines a class $e\in \KU^0(M\times S)$; see \cite[Section 3]{Markman-IntegralGeneratorPoisson} for the detailed construction of $e$ in the absence of universal sheaf. By construction, $\mathcal{E}$ is a real sheaf, hence the class $e\in \KU^0(M\times S)$ is the image of a class $\tilde{e}\in \KR^0(M\times S)$.
		
		In the sequel, for simplicity, we denote by $\KU^*\coloneqq \KU^0\oplus \KU^1$ and $\KR^*\coloneqq \KR^0\oplus \KR^1$.
	Let $p_1, p_2$ be the projections from $M\times S$ to $M$ and $S$ respectively. Since $\KU^*(S)$ is torsion-free (for any Poisson surface $S$), we have the K\"unneth decomposition for topological K-theory:
	$$\KU^0(M\times S)\cong \KU^0(M)\otimes \KU^0(S)\oplus  \KU^1(M)\otimes \KU^1(S).$$
	Via the above decomposition, we write the following equality in $\KU^0(M\times S)$:
	\begin{equation}\label{eqn:Classe}
	e=\sum_i p_1^*e_i\otimes p_2^*x^\vee_i,
	\end{equation}
	where $\{x_i\}_i$ is a basis of $\KU^*(S)$, and $e_i\in \KU^*(M)$. 
	
	By Markman \cite[Theorem 1]{Markman-IntegralGeneratorPoisson}, the cohomology ring $H^*(M, \Z)$ is torsion-free and generated by Chern classes $\{c_j(e_i)\}_{i,j}$, where $j$ runs through $\mathbb{N}+\frac{1}{2}$ if $e_i\in \KU^1(M)$; see \cite[Definition 19]{Markman-IntegralGeneratorPoisson} for the definition of Chern classes of elements in odd K-theory.
	Therefore, $H^*(M, \F2)$ is also generated by $\{c_j(e_i)\}_{i,j}$. We only need to prove that every generator $c_j(e_i)\in H^{2j}(M, \F2)$ has a preimage in $H^{2j}_G(M, \F2)$.
	
	Since the pairing 
	\begin{align*}
	\chi\colon \KU^*(S)\otimes \KU^*(S)&\to \Z\\
	(F_1, F_2)&\mapsto \chi(F_1^\vee\otimes F_2).
	\end{align*}
	is perfect (see \cite[Remark 18]{Markman-IntegralGeneratorPoisson}), one can find a dual basis $\{y_i\}_i$  of $\KU^*(S)$, i.e. $\chi(x_i^\vee \otimes y_j)=\delta_{ij}$.
	Since $\KR^*(S)\to \KU^*(S)$ is surjective by assumption, we have that for each $i$, there exists $\tilde{y_i}\in \KR^*(S)$ whose image in $\KU^*(S)$ is $y_i$. 
	
	Consider the element 
	\[\tilde{e_i}:=p_{1,!}(\tilde{e}\otimes p_2^*(\tilde{y_i}))\in \KR^*(M).\]
	The image of $\tilde{e_i}$ under the map $\KR^*(M)\to \KU^*(M)$ is computed as follows using \eqref{eqn:Classe}: 
	$$p_{1,!}(e\otimes p_2^*(y_i))=\sum_j p_{1,!}(p_1^*e_j\otimes p_2^*(x^\vee_j\otimes y_i))=\sum_j\chi(x_j^\vee\otimes y_i)e_j=e_i.$$
	We obtain that, for any $i$, the class $\tilde{e_i}\in \KR^*(M)$ is mapped to $e_i\in \KU^*(M)$. 
	Therefore, the equivariant Chern class $c_j^G(\tilde{e_i})\in H_G^{2j}(M, \F2)$ is mapped to the Chern class $c_j(e_i)\in H^{2j}(M, \F2)$ via the restriction map. 
	 As $\{c_j(e_i)\}$ are generators, the restriction map $H^*_G(M, \F2)\to H^*(M, \F2)$ is surjective, and $M$ is maximal by Proposition \ref{prop:MaximalCriterion}.
\end{proof}

	To provide examples where Theorem \ref{thm:PoissonSurface} applies, let us give the following generalization of Theorem \ref{thm:ModuliP2}:
	
\begin{cor}[Rational  real Poisson surface]
	\label{cor:RationalPoisson}
	Let $S$ be a Poisson surface with a real structure $\sigma$ such that it is $\R$-\emph{rational}. Let  $v$ be a primitive Mukai vector satisfying $\rk(v)>0$ and  $\sigma^*(c_1(v))=-c_1(v)$. Let $H$ be a real algebraic ample line bundle that is $v$-generic. If $(S, \sigma)$ is maximal, then $M_H(S, v)$, equipped with the natural real structure,  is also maximal.
\end{cor}
\begin{proof}
To apply Theorem \ref{thm:PoissonSurface}, we only need to check the $K$-maximality. This is verified in Lemma \ref{lemma:KMaxRationalSurface}.
\end{proof}

\section{Real motives, equivariant formality, and maximality}
\label{sec:Motive}
The notion of \textit{motive} was invented by Grothendieck around 1960 to unify various cohomology theories and to deal with problems involving algebraic cycles. It has become a standard and powerful tool nowadays throughout algebraic geometry. See \cite{AndreMotiveBook} for an introduction.  In this section, we would like to explain the relation between the maximality of real varieties and  their motives. Although not logically necessary for the results obtained in this paper, we highlight the following principle which underlies the whole paper:\\

\noindent\textbf{Principle:} The maximality of real varieties is a \textit{motivic} property called \textit{equivariant formality}, and new maximal varieties can be constructed as those \textit{motivated} by maximal ones. \\

For the precise meaning of the principle, see Definition \ref{def:MaximalMotive} and Corollary \ref{cor:MotivationMaximal}. We switch to the algebraic notation in this section: for  an algebraic  variety $X$ defined over $\R$, the complex analytic space is denoted by $X(\C):=X_{\C}^{an}$. In particular, $(X\times_{\R}Y)(\C)=X(\C)\times Y(\C)$.
\subsection{Motives}
Let us briefly recall some basic notions on motives; see \cite{AndreMotiveBook} for a detailed account.
\begin{defn}[Real motives] \label{def:Motives}
	Fix a commutative ring $F$ as the coefficient ring. In this paper,  $F=\Z$, $\Q$, or $\F2$, etc.\\
(i) The category of \textit{Chow motives} over the field $\R$ with coefficients in $F$, denoted by $\CHM(\R)_F$, is defined as follows:
	
- An object, called a Chow motive, is a triple $(X, p, n)$, where $n$ is an integer, $X$ is a smooth projective variety over $\R$ of dimension $d_X$, and  $p\in \CH^{d_X}(X\times_{\R} X)_F$ is an algebraic cycle (modulo rational equivalence) in $X\times_{\R} X$ with $F$-coefficients such that $p\circ p=p$, here $\circ$ denotes the composition of correspondences. 
 
- A morphism between two Chow motives $(X, p, n)$ and $(Y, q, m)$ is an element in the group 
 \[\Hom((X, p, n), (Y, q, m)):=q\circ \CH^{m-n+d_X}(X\times_{\R} Y)_F\circ p.\]
The composition of morphisms are given by composition of correspondences. 

- There is a natural (contravariant) functor of ``taking Chow motives":
\begin{align*}
\text{(Smooth projective $\R$-varieties, $\R$-morphisms)}^{\operatorname{op}} & \to\CHM(\R)_F\\
X&\mapsto  \h(X):=(X, \Delta_X, 0)\\
(f\colon Y\to X) & \mapsto \Gamma_f\in \CH^{d_X}(X\times_{\R} Y)_F
\end{align*}
where $\Delta_X\in \CH^{d_X}(X\times_{\R} X)_F$ is the class of the diagonal, and $\Gamma_f$ is the graph of $f$.  
  
The category $\CHM(\R)_F$ is an $F$-linear, idempotent complete, rigid  tensor category, with unit object $\mathds{1}=(\spec(\R), \Delta, 0)$, $(X, p, n)\otimes (Y, q, m):= (X\times_{\R} Y, p\times q, m+n)$ and $(X, p, n)^\vee:= (X, {}^tp, d_X-n)$, where ${}^tp$ is the transpose of $p$ (i.e. its image under the involution on $X\times_{\R} X$) swapping two factors. For an integer $n$, the Tate object is $\mathds{1}(n):=(\spec(\R), \Delta, n)$.

\noindent(ii) Similarly, by replacing the Chow groups by the groups of algebraic cycles modulo homological equivalence in the above definition, we obtain the category of \textit{homological} (or \textit{Grothendieck}) \textit{motives}, denoted by $\Mot(\R)_F$. The (contravariant) functor of taking homological motives is denoted by $X\mapsto M(X)$.

\noindent(iii) Since rational equivalence is finer than homological equivalence, we have a natural functor $\CHM(\R)_F\to \Mot(\R)_F$.
\end{defn}

\subsection{Realizations}
It is easy to check that any (generalized) cohomological theory (like Chow group $\CH^*$, cohomology $H^*$, equivariant cohomology  $H_G^*$, K-theory, topological K-theory, etc.) that behaves well with respect to correspondences (i.e.~with respect to pullbacks, pushforwards, and intersection products), can be extended to appropriate categories of motives, called \textit{realizations}. Let us mention some examples. Let $\Mod_F$ (resp. $\Mod_F^{\Z}$) be the category of (resp. graded) $F$-modules.

\begin{enumerate}
	\item (Chow group). Define
	\begin{align*}
		\CH^*\colon \CHM(\R)_F \to& \Mod_F^{\Z}\\
		(X, p, n)&\mapsto \im(p_*\colon \CH^{*+n}(X)_F\to \CH^{*+n}(X)_F).
	\end{align*}
It follows from the definition that the Chow group of a Chow motive $M$ is just a $\Hom$ group in the category $\CHM(\R)_F$:
$$\CH^i(M)_F=\Hom(\mathds{1}(-i), M).$$
\item (Equivariant cohomology). Let $G=\operatorname{Gal}(\mathbb{C}/\mathbb{R})$. For any integers $i, j$, we have
\begin{align*}
H_G^i(-, F(j))\colon \Mot(\R)_F &\to \Mod_F\\
(X, p, n)&\mapsto \im(\cl_G(p)_*\colon H^{i+2n}_G(X(\C), F(j+n))\to H^{i+2n}_G(X(\C), F(j+n))),
\end{align*}
where $F(j)=\sqrt{-1}^j\Z$ as $G$-module.\\
Here we used the existence of the equivariant cycle class map: for any smooth projective real variety $W$, we have 
\[\cl_G\colon \CH^*(W)_F\to H^{2*}_G(W(\C), F(*)),\]
which is  compatible with the functorialities and cup-product of equivariant cohomology, thanks to Krasnov \cite{Krasnov-CharClass}.
Let us also mention that when $F$ is torsion, $H_G^i(-, F(j))$ is canonically identified with the \'etale cohomology $H_{\et}^i(-, F(j))$ (see \cite[Corollary 15.3.1]{Sheiderer-LNM-Real&EtaleCohom}), and the standard compatibility with \'etale cycle class map yields the \'etale realization functor. 

\item (Cohomology of complexification). For motives over $\C$, we have the realization of singular cohomology:
\begin{align*}
	H^*\colon \Mot(\C)_F&\to \Mod_F^{\Z}\\
(X, p, n)& \mapsto \im(\cl(p)_*\colon H^{*+2n}(X(\C), F)\to H^{*+2n}(X(\C), F)),
\end{align*}
where $\cl$ is the usual (complex) cycle class map. 
By composing with the complexification functor $-_{\C}\colon \Mot(\R)_F\to \Mot(\C)_F$, we get the following realization, still denoted by $H$:
\begin{align*}
	H^*\colon \Mot(\R)_F&\to \Mod_F^{\Z}\\
	M=(X, p, n)& \mapsto H^*(M_{\C}):= \im(\cl(p_{\C})_*\colon H^{*+2n}(X(\C), F)\to H^{*+2n}(X(\C), F)).
\end{align*}
In fact, when $F=\Z$ or $\Q$, we can further enhance the realization functor into a functor with target the category of (polarizable) Hodge structures, still denote by $H^*$:
\begin{equation*}
	H^*\colon \Mot(\R)_F\to \operatorname{HS}_F^{\Z}
\end{equation*}

\item (Cohomology of real loci). To each real variety $X$, one can associate the cohomology of its real locus $H^*(X(\R), \F2)$. We claim that this extends to a realization functor:
\begin{align*}
	H\!\R^*\colon \Mot(\R)_{\F2}&\to \Mod_{\F2}^{\Z}\\
	(X, p, n)& \mapsto \im(\cl_R(p)_*\colon H^{*+n}(X(\R), \F2)\to H^{*+n}(X(\R), \F2)).
\end{align*}
Here we used that for any smooth projective real variety $W$, we have the Borel--Haefliger real cycle class map  \cite{BorelHaefliger}:
\[\cl_R\colon \CH^*(W)\to H^*(W(\R), \F2),\]
that is compatible with  pullbacks, pushforwards and the intersection product. The latter fact can be justified as follows: thanks to Krasnov \cite{Krasnov-94}, the Borel--Haefliger cycle class map is the composition of the equivariant cycle class map $\cl_G$ with the map of restricting to the real locus constructed in \cite{Krasnov-94}: 
$$H^{2*}_G(W, \F2)\to H^*(W(\R), \F2).$$ The compatibility with respect to the pullbacks and the intersection product is easy to check while the compatibility with respect to  pushforwards can be found in \cite[Theorem 1.21 and Proposition 1.22]{BenoistWittenberg-RealIHC-1}.

\item We have various natural transformations between these realization functors on motives. For example, for any Chow motive $M$, we have
\begin{itemize}
\item the restriction map: $H^i_G(M, F(j))\to H^i(M_{\C}, F)$;
\item the equivariant cycle class map $\cl_G\colon \CH^*(M)\to H^{2*}_G(M, \Z(*))$;
\item Krasnov's map of restricting to real loci: $H^{2*}_G(M, \F2)\to H\!\R^*(M, \F2)$;
\item  composing the previous two morphisms, we get Borel--Haefliger's real cycle class map for real motives $\cl_R\colon \CH^*(M)\to H\!\R^*(M, \F2)$;
\end{itemize}

\end{enumerate}

We can now generalize the notion of maximality as follows, in view of Proposition \ref{prop:MaximalCriterion}. 
\begin{defn}[Equivariantly formal real motives]\label{def:MaximalMotive}
	A homological motive $M\in \Mot(\R)_{\F2}$ is called \textit{equivariantly formal}, if the restriction morphism 
	$H^*_G(M, \F2)\to H^*(M_{\C}, \F2)$ is surjective. 
\end{defn}

By Proposition \ref{prop:MaximalCriterion}, a smooth projective real variety $X$ is maximal if and only if its homological motive  $M(X)\in \Mot(\R)_{\F2}$ is equivariantly formal. 
Of course, we have the following more naive generalization of maximality (and Hodge expressivity) to motives.
\begin{defn}[Maximal and Hodge expressive motives]
	A homological motive $M\in \Mot(\R)_{\F2}$ is called \textit{maximal}, if 
	$$\dim H\!\R^*(M, \F2)= \dim H^*(M_{\C}, \F2).$$
Similarly, $M\in \Mot(\R)$ is called \textit{Hodge expressive}, if $H^*(M, \Z)$ is torsion-free and for any $p\in \Z$,
	$$\dim H\!\R^p(M, \F2)=\sum_{q} \dim H^{p,q}(M),$$ where $H^{p,q}(M)$ is the $(p,q)$-component of the Hodge structure $H^*(M_{\C}, \Q)$.
\end{defn}

\begin{rmk}
	The notions of maximality and Hodge expressivity for real motives defined above do not seem very useful, since it is a priori not inherited by submotives;
	see Remark \ref{rmk:SummandHE}.
\end{rmk}

\subsection{Inheriting maximality by motivation}
Recall that the category $\Mot(\R)_{\F2}$ is a tensor category. 
We first check that the standard tensor operations in $\Mot(\R)_{\F2}$ preserve equivariant formality. 

It is clear from Definition \ref{def:MaximalMotive} that the direct sum of some equivariantly formal motives is equivariantly formal,  and a direct summand of an equivariantly formal motive is equivariantly formal, and one can check directly (using the maximality of $\PP^n_{\R}$ for example) that the Tate motives $\mathds{1}(n)$ are equivariantly formal for all $n\in \Z$. 

Next, recall that the tensor product of motives is given essentially by the fiber product (see Definition \ref{def:Motives}). Similarly to Lemma \ref{lemma:Product}, equivariant formality is preserved by tensor product of motives.
\begin{lemma}[Tensor product]
	\label{lemma:TensorMaxMotive}
If $M_1, \cdots, M_n\in \Mot(\R)_{\F2}$ are equivariantly formal real motives, then $M_1\otimes \cdots\otimes M_n$ is also equivariantly formal. In particular, Tate twists of equivariantly formal real motives are equivariantly formal.
\end{lemma}
\begin{proof}
By induction, it is enough to show the case where $n=2$. Consider the following commutative diagram, where we omit the coefficient $\F2$:
\begin{equation}
\xymatrix{
H^*_G(M_1)\otimes H^*_G(M_2) \ar[d]\ar@{->>}[r]& H^*({M_1}_{\C})\otimes H^*({M_2}_{\C})\ar[d]^{\cong}\\
H^*_{G}(M_1\otimes M_2) \ar[r]& H^*({M_1}_{\C}\otimes {M_2}_{\C}),
}
\end{equation}
where the horizontal arrows are restriction maps, the right vertical arrow is given by the exterior cup product of singular cohomology, the left vertical arrow is the composition of the exterior cup product for equivariant cohomology followed by the restriction map from $G\times G$ to $\Delta_G$:
$$H^*_G(M_1)\otimes H^*_G(M_2)\to H^*_{G\times G}(M_1\otimes M_2) \to H^*_{G}(M_1\otimes M_2),$$ (or equivalently, since we are using $\F2$-coefficients, we can identify the equivariant realization with the \'etale realization, the left arrow is given by the exterior cup product for \'etale cohomology). 

Since the right vertical arrow is an isomorphism by the K\"unneth formula and the top arrow is surjective by assumption, we can conclude that the bottom arrow is also surjective. 
\end{proof}

%

Fix a category of motives $\mathcal{C} =\CHM(\R)_F$ or $\Mot(\R)_F$. For a collection of (Chow or homological) motives $\{M_i\}_{i\in I}$, we denote by $\langle M_i; i\in I\rangle^{\otimes}$ the thick tensor subcategory generated by the $M_i$'s, that is, the smallest $F$-linear subcategory of $\mathcal{C}$  containing all $M_i$'s and closed under isomorphisms, direct sums, tensor products, direct summands, and Tate twists. By Lemma \ref{lemma:TensorMaxMotive}, if all $M_i$ are equivariantly formal motives, then every object in the category  $\langle M_i; i\in I\rangle^{\otimes}_{\Mot(\R)_{\F2}}$ is equivariantly formal. 

\begin{defn}[Motivation]
	A real variety $Y$ is called \textit{motivated} by a collection of smooth projective real varieties $\{X_i\}_{i\in I}$, if its homological motive with $\F2$-coefficients $M(Y)$ lies in the subcategory
	$\langle M(X_i); i\in I\rangle^{\otimes}_{\Mot(\R)_{\F2}}$.
\end{defn}

The main point of this section is the following observation:
\begin{cor}\label{cor:MotivationMaximal}
	 A real variety $Y$ motivated by a collection of maximal real varieties $\{X_i\}_{i\in I}$  is maximal. 
\end{cor}
\begin{proof}
	It is almost evident:  if $Y$ is motivated by $\{X_i\}_{i\in I}$, that is, in the category $\Mot(\R)_{\F2}$, 
	\[M(Y)\in \langle M(X_i); i\in I\rangle^{\otimes }.\]
	Since $M(X_i)$ is equivariantly formal for any $i$, Lemma \ref{lemma:TensorMaxMotive} (and the remarks preceding it) implies that $M(Y)$ is equivariantly formal. Hence $Y$ is maximal by Proposition \ref{prop:MaximalCriterion}.
\end{proof}

\begin{rmk}[Hodge expressivity]
	\label{rmk:SummandHE}
	Maximality, as well as its strengthening Hodge expressivity (see \cite{BrugalleSchaffhauser})
are notions in terms of realizations (Hodge numbers, Betti numbers of real loci etc.), it is clear that the tensor product of two maximal (resp.~Hodge expressive) motives is again maximal  (resp.~Hodge expressive). However, it is unclear if a submotive of a maximal or Hodge expressive motive inherits the same property. Hence, a priori, we do not have the Hodge expressive analogue of Corollary \ref{cor:MotivationMaximal}.
\end{rmk}

\subsection{Applications}
\label{subsec:Applications}
Despite the formal nature of Corollary \ref{cor:MotivationMaximal},
many results obtained in the previous sections on various constructions of maximal real algebraic varieties can be viewed as instances of Corollary \ref{cor:MotivationMaximal}. Let us list them again by providing either an alternative motivic proof using Corollary \ref{cor:MotivationMaximal}, or a motivic explanation of the proof, or a conjectural motivic strengthening, in the light of Corollary \ref{cor:MotivationMaximal}. Notation is retained from the statements in the previous sections, and in the motivic results, the base field $\R$ can often be replaced by any field.
\begin{enumerate}
	\item (Projective bundle). Proposition \ref{prop:ProjBun} can be proved directly by using the projective bundle formula for motives, which holds even integrally for Chow motives (\cite[\S 4.3.2]{AndreMotiveBook}):
	$$\h(\PP_X(E))\simeq \h(X)\oplus \h(X)(-1)\oplus \cdots\oplus \h(X)(-r) \quad \quad \text{ in } \CHM(\R).$$
	Similarly, the claim in Remark \ref{rmk:FlagBundles} on flag bundles follows from the flag bundle formula for motives:
	$$\h(\operatorname{Fl}_X(E))\simeq \h(X)\otimes \h(\operatorname{Fl(\mathbb{A}^{r+1})}).$$
	\item (Blow-up). Proposition \ref{prop:Blowup} can be seen directly by using the blow-up formula for motives, which also holds integrally for Chow motives (\cite[\S 4.3.2]{AndreMotiveBook}):
	$$\h(\Bl_YX)\simeq \h(X)\oplus \h(Y)(-1)\oplus \cdots\oplus \h(Y)(-(c-1)) \quad \quad \text{ in } \CHM(\R).$$
	\item (Compactification of configuration space). Proposition \ref{prop:Configuration} can be deduced from the fact that the Fulton--MacPherson compactification of the configuration space $X[n]$ is motivated by $X$ (even integrally for Chow motives). 
	\item (Surjection of odd degree). The proof of Proposition \ref{prop:SurjectionOddDegree} actually shows that for a generically finite surjection of odd degree $Y\to X$,  $X$ is motivated (even for Chow motives with $\F2$-coefficients) by $Y$.
	\item (Moduli of vector bundles on curves). For the moduli space $M_C(n,d)$ in Theorem \ref{thm:VBAC}, the main ingredient in our proof is Atiyah--Bott's Theorem \ref{thm:AtiyahBott}, which suggests that the integral Chow motive of $M_C(n,d)$ might be motivated by $C$ and its symmetric powers. It was known to hold with rational coefficients \cite[Theorem 4.5]{dB_motive_moduli_vb}, \cite{Beauville_diag}, \cite[Remark 2.2]{Bulles}, \cite[Proposition 4.1]{FHPL-rank2}. In the up-coming \cite{HoskinsPepin-IntegralMotVBAC}, it is shown that the \textit{integral Chow motive of $M_C(n,d)$ is motivated by the symmetric powers of $C$}. Using Franz's theorem \ref{thm:Franz}, we get yet another alternative proof of Theorem \ref{thm:VBAC}. 
	\item (Moduli of parabolic bundles). Using the language of motives, Corollary \ref{cor:Parabolic} can be proved by induction and by saying that \cite[Corollary 5.16 and Corollary 5.19]{FHPL-rank2} imply that $M_{C, D}^{\alpha}(n, d, \underline{1})$ is motivated by $M_C(n, d)$ and all $M_{C, D}^{\alpha}(n', d', \underline{1})$ with $(n', d', \underline{1})$ smaller invariants. 
	\item (Moduli of Higgs bundles on curves). The main ingredient in the proof of  Theorem \ref{thm:Higgs} for the maximality of the moduli space $H_C(n, d)$ is Theorem \ref{thm:GeneratorHiggs} (a variant Markman's result \cite{Markman-IntegralGeneratorPoisson}). One can expect that $H_C(n, d)$ is motivated by all symmetric powers of $C$. This is only established with rational coefficients in \cite[Theorem 4.1]{HoskinsPepin-Higgs}.
	\item (Hilbert square and cube). In Theorem \ref{thm:HilbertSquareCube}, we actually showed that $X^{[2,3]}$ and $X^{[3]}$ are motivated (for Chow motives with $\F2$-coefficients) by $X^{[2]}$ and $X$.  
	\item (Cubic threefolds). Theorem \ref{thm:Cubic3folds} can be proven directly by using the fact that for a cubic threefolds $X$, its Hilbert square $X^{[2]}$ is motivated by $X$ and its Fano surface of lines $F(X)$, thanks to \cite[Corollary 18]{BelmansFuRaedschelders}.
	\item (Hilbert schemes of surfaces). For $S^{[n]}$, the key ingredient towards Theorem \ref{thm:HilbertPower} is Li--Qin's Theorem \ref{thm:LiQin}, which leads to the challenge of understanding the motive of $S^{[n]}$ in terms of the motive of $S$. The author has no idea for integral coefficients. But with rational coefficients, de Cataldo--Migliorini \cite{deCataldoMigliorini-Hilbert} showed that $S^{[n]}$ is motivated by $S$.
	\item (Moduli of sheaves on $\PP^2$).  Ellingsrud--Str{\o}mme \cite{EllingsrudStromme} actually shows that the Chow motive of $M(r, c_1, c_2)$ is a direct sum of Tate motives in $\CHM(\R)$, from which Theorem \ref{thm:ModuliP2} can be immediately deduced, by Corollary \ref{cor:MotivationMaximal}.
	\item (Moduli of sheaves on Poisson surfaces). In view of the result of Markman \cite{Markman-IntegralGeneratorPoisson}, we would like to make the following conjecture, with the hope of having more cases than in Theorem \ref{thm:PoissonSurface}. 
	\begin{conj}
		The integral Chow motive of  $M_H(S, v)$, the moduli space of stable sheaves on a Poisson surface $S$ with primitive Mukai vector $v$ and $v$-generic stability condition $H$, is in the tensor category generated by the motives of Hilbert schemes $\{S^{[n]}\}_n$.
	\end{conj}
This conjecture is known with rational coefficients, thanks to B\"ulles \cite{Bulles}.
\end{enumerate}


To further illustrate the power of this motivic point of view, we provide some more applications. 
Firstly,  we have the following complement to Theorem \ref{thm:HilbertPower}.

\begin{thm}[Hilbert schemes of surfaces with Tate motives]
	\label{thm:SurfaceTate}
	Let $S$ be a smooth projective real surface. If its integral Chow motive $\h(S)\in \CHM(\R)$ is a direct sum of Tate motives, then the Hilbert scheme $S^{[n]}$ is maximal for any $n\geq 1$.
\end{thm}
\begin{proof}
	By Totaro \cite[Theorem 4.1]{TotaroHilbn}, $\h(S^{[n]})$ is also a direct sum of Tate motives. Since Tate motives are equivariantly formal, $S^{[n]}$ is maximal.
\end{proof}
\begin{rmk}
	The above theorem generalizes Corollary \ref{cor:RealRationalSurface}, since $\R$-rational real surfaces all have Tate motives. This can be seen from the minimal model theory for real surfaces, the classification of rational real surfaces, and the blow-up formula (see the proof of Lemma \ref{lemma:KMaxRationalSurface}).
\end{rmk}

Secondly, in the direction of \cite{Kharlamov-Rasdeaconu-HilbertSquare}, we show that cubic fourfolds provide more examples of maximal varieties with non-maximal Hilbert square. This contrasts to the case of cubic threefolds in Theorem \ref{thm:Cubic3folds}. 
\begin{thm}[Hilbert squares of cubic fourfolds]
	\label{thm:NonMaxHilbCubic4}
	Let $X$ be a real smooth cubic fourfold. Assume that in Finashin-Kharlamov's classification \cite{FinashinKharlamov}, $X$ does not belong to the regular class  corresponding to K3 surfaces with 10 spheres as real locus, then the Hilbert square $X^{[2]}$ is not maximal. In particular, the Hilbert square of a maximal smooth real cubic fourfold is not maximal. 
\end{thm}
\begin{proof}
	In \cite[Theorem 7.7]{Kharlamov-Rasdeaconu-HilbertSquare} it is shown that for such a cubic fourfold $X$, the Fano variety of lines $F(X)$ is not maximal. However, by the motivic Galkin--Shinder relation established in \cite[Corollary 18]{BelmansFuRaedschelders}, $F(X)$ is motivated by $X^{[2]}$. Therefore, Corollary \ref{cor:MotivationMaximal} implies that $X^{[2]}$ cannot be maximal. Finally, note that a maximal cubic fourfold does not belong to the regular class corresponding to real K3 surfaces with 10 spheres as real locus. 
\end{proof}

%


\bibliographystyle{amsplain}
\bibliography{references}

\medskip \medskip

\noindent{Universit\'e de Strasbourg, Institut de recherche mathématique avancée (IRMA)  $\&$ Institut d’études avancées de l'université de Strasbourg (USIAS), France} 

\medskip \noindent{\texttt{lie.fu@math.unistra.fr}}

\end{document}